\documentclass[10pt,reqno]{amsart}

\usepackage{fullpage}

\usepackage{amsmath,amsthm,amssymb}
\usepackage{hyperref}

\numberwithin{equation}{section}

\newtheorem{theorem}{Theorem}[section]
\newtheorem{proposition}[theorem]{Proposition}
\newtheorem{lemma}[theorem]{Lemma}
\newtheorem{corollary}[theorem]{Corollary}
\newtheorem{application}[theorem]{Application}
\theoremstyle{definition}

\newcommand{\R}{\mathbb R}
\newcommand{\N}{\mathbb N}
\newcommand{\Prob}{\mathbb P}
\newcommand{\E}{\mathbb E}
\newcommand{\1}{\mathbf 1}
\newcommand{\abs}[1]{\left|#1\right|}
\newcommand{\norm}[1]{\left\|#1\right\|}
\newcommand{\osc}{\operatorname{osc}}
\newcommand{\esssup}{\operatorname*{ess sup}}
\newcommand{\essinf}{\operatorname*{ess inf}}
\newcommand{\cQ}{\mathcal Q}
\newcommand{\cB}{\mathcal B}
\newcommand{\cC}{\mathcal C}
\newcommand{\cA}{\mathcal A}
\newcommand{\cI}{\mathcal I}
\newcommand{\cJ}{\mathcal J}
\newcommand{\Adm}{\operatorname{Adm}}
\newcommand{\vol}{\operatorname{vol}}

\title[Product-profile anti-concentration]{Product-profile anti-concentration for block-structured multi-affine polynomials}

\author{Evgeny Abakumov}
\address{Univ Gustave Eiffel, Univ Paris Est Creteil, CNRS, LAMA UMR 8050, F-77447 Marne-la-Vall\'ee, France}
\email{evgueni.abakoumov@univ-eiffel.fr}

\author{Omer Friedland}
\address{Institut de Math\'ematiques de Jussieu-Paris Rive Gauche (IMJ-PRG), Sorbonne Universit\'e, Campus Pierre et Marie Curie, 4 place Jussieu, Bo\^ite Courrier 247, 75252 Paris Cedex 5, France}
\email{omer.friedland@imj-prg.fr}

\author{Yosef Yomdin}
\address{Department of Mathematics, Faculty of Mathematics and Computer Science, Weizmann Institute of Science, Jacob Ziskind Building, Room 153, Rehovot, Israel}
\email{yosef.yomdin@weizmann.ac.il}


\dedicatory{In memory of Yosef Yomdin}

\date{}

\begin{document}

\begin{abstract}
We establish product-profile anti-concentration, density, and Remez
estimates for multi-affine polynomials on the cube. Let $X$ be uniformly distributed on $[0,1]^n$, and let $P:[0,1]^n\to\R$ be a nonconstant multi-affine
polynomial of exact degree $d$, with all variables active. For a
partition $\cB$ of the variables such that no monomial contains two
variables from the same block, define
$$
w_d(\cB) = 
\sum_{\cI\subseteq\cB, \abs{\cI} = d-1}
\left(\prod_{B\in\cI}\sqrt{\abs B}\right)
\sqrt{\sum_{C\in\cB\setminus\cI}\abs C}.
$$
We prove the all-center small-ball estimate
$$
\sup_{u\in\R}
\Prob\{\abs{P(X)-u}\le\rho\} \le
\Phi_d\left(
C_dw_d(\cB)\frac{\rho}{\osc(P)}
\right),
$$
where $\osc(P)$ is the length of the range of $P$ and
$\Phi_d(t) = t\sum_{j = 0}^{d-1}\log^j(1/t)/j!$ on $(0,1]$, capped at one.
Every multi-affine polynomial admits the singleton partition, which
yields the universal scale $n^{d-1/2}$; an admissible partition into at
most $q$ blocks yields the improved scale
$q^{(d-1)/2}n^{d/2}$. If $\cB$ has exactly $d$ blocks, then
$$
w_d(\cB) = d\prod_{B\in\cB}\sqrt{\abs B},
$$
and products of centered block averages show that both the small-ball
profile and the dependence on the full block-size vector are optimal,
up to constants depending only on $d$.

We also prove that $P(X)$ has a density $f_P$ satisfying
$$
\norm{f_P}_{L^p(\R)} \le
C_dp^{d-1}
\left(
\frac{w_d(\cB)}{\osc(P)}
\right)^{1-1/p},
\quad 1<p<\infty,
$$
with matching $p$- and block-scale growth for $p\ge2$ on the
block-product models. As an application, we derive
translation-invariant quotient Remez inequalities with the same
structural scale. The proof combines weighted block selection, affine
cube slicing, degree-lowering contractions, and the exact recursion
underlying the product profile.
\end{abstract}

\maketitle

\section{Introduction} \label{sec:introduction}

\subsection{The problem}

Let $\mu_n$ denote normalized Lebesgue measure on $[0,1]^n$. Let
$X = (X_1,\ldots,X_n)$ be a random point sampled uniformly from the cube,
meaning that
$$
\Prob\{X\in A\} = \mu_n(A)
$$
for every Borel set $A\subseteq[0,1]^n$; we write $X\sim\mu_n$.
For a nonconstant polynomial $P:[0,1]^n\to\R$, the notation $P(X)$
denotes the real-valued random variable obtained by evaluating $P$ at
$X$. We seek upper bounds for the concentration function
\begin{align*}
\cQ_P(\rho) & = 
\sup_{u\in\R}\Prob\{\abs{P(X)-u}\le\rho\} \\
& = \sup_{u\in\R}
\mu_n\{x\in[0,1]^n:\abs{P(x)-u}\le\rho\},
\quad \rho\ge0,
\end{align*}
in terms of the relative radius $\rho/\osc(P)$, where
$$
\osc(P) = \max_{x\in[0,1]^n}P(x) - \min_{x\in[0,1]^n}P(x).
$$

For arbitrary degree-$d$ polynomials, the natural general estimates have a power profile of order $t^{1/d}$. Products behave differently: if $U_1,\ldots,U_d$ are independent and uniform on $[0,1]$, then
$$
\Prob\{U_1\cdots U_d\le t\} = \Phi_d(t),\quad \Phi_d(t) = t\sum_{j = 0}^{d-1}\frac{\log^j(1/t)}{j!},\quad 0<t\le1.
$$
We set $\Phi_d(0) = 0$ and $\Phi_d(t) = 1$ for $t\ge1$.
Thus
\begin{align} \label{eq:Phi-asymp-intro}
\Phi_d(t)\sim\frac{t\log^{d-1}(1/t)}{(d-1)!}\quad(t\downarrow0).
\end{align} The question addressed here is whether this product profile persists for a broad class of nonhomogeneous multi-affine polynomials, and whether the relevant scale can be expressed in terms of their block structure and global oscillation.

\subsection{History and context}

Remez inequalities control a polynomial on a full set by its size on a subset of positive measure. Their multivariate form, due to Brudnyi and Ganzburg, applies to arbitrary polynomials on convex bodies \cite{BrudnyiGanzburg73}. In probability, Nazarov, Sodin, and Volberg developed a dimension-free geometric approach yielding sharp two-sided distribution inequalities for polynomials on convex bodies, normalized by a fixed quantile, together with related moment estimates \cite{NazarovSodinVolberg03}. They note that their upper-tail estimate is essentially due to Bourgain. Closely related dimension-free distribution and $L^q$-norm inequalities were developed by Carbery and Wright under general log-concave measures, including Banach-valued polynomials \cite{CarberyWright01}.

For polynomials of bounded individual degree, logarithmically corrected linear behavior was already known in other settings. Kosov proved such regularity for Gaussian polynomial images under restrictions on the powers of the variables \cite{KosovIndividualDegree22}. In published work, Glazer and Mikulincer proved all-center anti-concentration under an aggregate coefficient normalization, with the generic power profile \cite{GlazerMikulincer22}. Hu and Paouris obtained logarithmically corrected small-ball estimates for projections of simple random tensors, including an overlap with homogeneous scalar block-multilinear forms \cite{HuPaouris26}. Published work of Kosov gives fractional Besov regularity for polynomial images under much broader measure-theoretic assumptions \cite{Kosov18,Kosov22}.

Two recent preprints are particularly relevant. Kosov's product-measure preprint yields the product-order profile for multi-affine polynomials under a maximum top-degree coefficient normalization \cite{Kosov25}. The recent preprint of Glazer and Mikulincer treats $L^p$ balls and symmetric measures \cite{GlazerMikulincer26}.

The product profile is therefore not new in isolation. The point of the present paper is to combine it with global oscillation normalization, an explicit weighted scale attached to an admissible block partition, matching unequal-block examples, finite-$p$ density bounds with the same scale, and a quotient-form Remez application.

\subsection{Main results}

Let $\N = \{1,2,\ldots\}$ and $[n] = \{1,\ldots,n\}$. A polynomial $P:[0,1]^n\to\R$ is \emph{multi-affine} if it is affine in each coordinate separately. Equivalently,
\begin{align} \label{eq:expansion}
P(x) = \sum_{S\subseteq[n]}a_Sm_S(x),\quad m_S(x) = \prod_{i\in S}x_i,
\end{align}
with $m_\varnothing = 1$. Its degree is
$$
\deg(P) = \max\{\abs S:a_S\ne0\}.
$$
A variable is \emph{active} if $P$ depends on it. In the main statements all variables are active, so $n$ is both the
dimension of the cube and the number of active variables, while
$d=\deg(P)$ is the actual degree. Constants denoted by $c_d,C_d$ are positive and depend only on $d$; their values may change from line to line and are not optimized. We write $A\lesssim_dB$ if $A\le C_dB$, and $A\asymp_dB$ if both $A\lesssim_dB$ and $B\lesssim_dA$. The notation $\norm P_\infty$ means the supremum norm on the relevant cube.

A partition $\cB$ of the active variables is \emph{admissible for $P$} if
$$
\abs{S\cap B}\le1
$$
for every $B\in\cB$ and every $S$ with $a_S\ne0$. We write $\Adm(P)$ for the set of admissible partitions. Equivalently, after the coordinates outside $B$ are fixed, the dependence on the whole vector $x_B$ is affine. Every multi-affine polynomial admits the singleton partition. Nontrivial examples need not be products: the polynomial $x_1x_2+x_3x_4$ admits $\{\{1,3\},\{2,4\}\}$, whereas $\{\{1,2\},\{3,4\}\}$ is not admissible.

The block-deletion recursion in the proof leads to a weighted sum over choices of $r-1$ blocks, with each term recording the product of their square-root sizes and the square-root size of the remaining coordinates. For a finite partition $\cC$ and $1\le r\le\#\cC$, define
\begin{align} \label{eq:w-scale}
w_r(\cC) = \sum_{{\cI\subseteq\cC, \abs{\cI} = r-1}}\left(\prod_{B\in\cI}\sqrt{\abs B}\right)\sqrt{\sum_{C\in\cC\setminus\cI}\abs C}.
\end{align}
If $P$ has exact degree $d$ and $\cB\in\Adm(P)$, then $\#\cB\ge d$, so $w_d(\cB)$ is defined.

The scale interpolates between several familiar regimes. For the
singleton partition,
$$
w_d(\cB) = \binom{n}{d-1}\sqrt{n-d+1} \asymp_d n^{d-1/2}.
$$
If $\cB$ has at most $q$ blocks, then
$$
w_d(\cB)\lesssim_d q^{(d-1)/2}n^{d/2},
$$
while for a partition into exactly $d$ blocks,
$$
w_d(\cB) = d\prod_{B\in\cB}\sqrt{\abs B}.
$$

\begin{theorem}[Small-ball inequality] \label{thm:main-smallball}
Let $P:[0,1]^n\to\R$ be a nonconstant multi-affine polynomial of exact degree $d$, assume that all variables are active, and let $\cB\in\Adm(P)$. Then, for every $\rho\ge0$,
\begin{align} \label{eq:main-weighted}
\cQ_P(\rho)\le\Phi_d\left(C_dw_d(\cB)\frac{\rho}{\osc(P)}\right).
\end{align}
\end{theorem}

The second result identifies the profile and the complete block-size scale when the partition has the minimal possible number of blocks. Let $\mathbf m = (m_1,\ldots,m_d)\in\N^d$. Choose disjoint blocks $B_1,\ldots,B_d$ with $\abs{B_k} = m_k$, and put
$$
N_{\mathbf m} = m_1+\cdots+m_d,\quad M_{\mathbf m} = \prod_{k = 1}^d\sqrt{m_k},
$$
and define
\begin{align} \label{eq:unequal-product}
P_{\mathbf m,d}(x) = \prod_{k = 1}^d\left(\frac2{m_k}\sum_{i\in B_k}x_i-1\right),\quad x\in[0,1]^{N_{\mathbf m}}.
\end{align}

\begin{theorem}[Block-product models] \label{thm:block-products}
For every $d\ge1$, there are constants $c_d,C_d,s_d>0$, depending only on $d$, such that, for every $\mathbf m\in\N^d$ and $0\le s\le s_d$,
\begin{align} \label{eq:block-products-two-sided}
\Phi_d(c_ds)\le\Prob\left\{\abs{P_{\mathbf m,d}(X)}\le\frac{s}{M_{\mathbf m}}\right\}\le\Phi_d(C_ds),
\end{align}
where $X$ is uniform on $[0,1]^{N_{\mathbf m}}$. The upper bound holds for every $s\ge0$, and
\begin{align} \label{eq:block-product-osc}
\norm{P_{\mathbf m,d}}_\infty = 1,\quad \osc(P_{\mathbf m,d}) = 2.
\end{align}
The scale $M_{\mathbf m}$ is optimal for prescribed admissible $d$-block partitions in the following sense. Suppose that $\Gamma:\N^d\to(0,\infty)$ and $K_d>0$ satisfy
$$
\cQ_P(\rho)\le\Phi_d\left(K_d\Gamma(\mathbf m)\frac{\rho}{\osc(P)}\right)
$$
for every $\mathbf m\in\N^d$, every exact degree-$d$ multi-affine polynomial $P:[0,1]^{N_{\mathbf m}}\to\R$ with all variables active and with a prescribed admissible $d$-block partition of sizes $m_1,\ldots,m_d$, and every $\rho\ge0$. Then
$$
\Gamma(\mathbf m)\ge c_dK_d^{-1}M_{\mathbf m},\quad \mathbf m\in\N^d,
$$
where $c_d>0$ depends only on $d$.

Moreover, fix $q\ge d$. If constants $K_{d,q}$ and $\alpha\ge0$ satisfy
\begin{align} \label{eq:hypothetical-alpha}
\cQ_P(\rho)\le\Phi_d\left(K_{d,q}n^\alpha\frac{\rho}{\osc(P)}\right)
\end{align}
for every exact degree-$d$ multi-affine polynomial with $n$ active variables and an admissible partition into at most $q$ blocks, then
\begin{align} \label{eq:alpha-barrier}
\alpha\ge\frac d2.
\end{align}
\end{theorem}

The same weighted contraction argument controls the complete image density.

\begin{theorem}[Density estimate] \label{thm:main-density}
Under the assumptions of Theorem~\ref{thm:main-smallball}, the random
variable $P(X)$ has an absolutely continuous distribution. Denote its
density by $f_P$. Then, for every $1<p<\infty$,
\begin{align} \label{eq:main-density}
\norm{f_P}_{L^p(\R)}\le C_dp^{d-1}\left(\frac{w_d(\cB)}{\osc(P)}\right)^{1-1/p}.
\end{align}
For the block products in \eqref{eq:unequal-product}, let $f_{\mathbf m,d}$ be the density of $P_{\mathbf m,d}(X)$ for $X\sim\mu_{N_{\mathbf m}}$. Then
$$
\norm{f_{\mathbf m,d}}_{L^p(\R)}\asymp_dp^{d-1}M_{\mathbf m}^{1-1/p},\quad p\ge2,
$$
uniformly over $\mathbf m\in\N^d$; see Proposition~\ref{prop:block-density}.
\end{theorem}

Finally, the all-center estimate has a deterministic Remez consequence. By Lemma~\ref{lem:Phi-calculus} below, $\Phi_d$ is a continuous strictly increasing bijection of $[0,1]$ onto itself; we denote its inverse by $\Phi_d^{-1}$. For a measurable set $\Omega\subseteq[0,1]^n$ of measure $\theta>0$, put
\begin{align} \label{eq:rOmega}
r_\Omega(P) = \inf_{c\in\R}\esssup_{x\in\Omega}\abs{P(x)-c}.
\end{align}

\begin{application}[Translation-invariant Remez inequality] \label{thm:main-remez}
Under the assumptions of Theorem~\ref{thm:main-smallball}, if $\mu_n(\Omega) = \theta>0$, then
\begin{align} \label{eq:main-remez}
\osc(P)\le\frac{C_dw_d(\cB)}{\Phi_d^{-1}(\theta)}r_\Omega(P).
\end{align}
Equivalently, after changing the degree-dependent constant,
\begin{align} \label{eq:main-remez-log}
\osc(P)\le C_dw_d(\cB)\frac{\log^{d-1}(e/\theta)}{\theta}r_\Omega(P).
\end{align}
For the block-product sublevel sets, the dependence on the full block-size vector is attained up to constants depending only on $d$.
\end{application}

\subsection{Comparison with related work}

The preceding theorems combine a profile and a scale that appear separately in the closest coefficient-normalized results. Put
$$
A_\infty(P) = \max_{\abs S = d}\abs{a_S},\quad A_2(P) = \left(\sum_{\abs S = d}\abs{a_S}^2\right)^{1/2}.
$$
Kosov's recent product-measure preprint, specialized to individual degree one, gives the product-order profile under normalization by $A_\infty(P)$ \cite{Kosov25}. The published theorem of Glazer and Mikulincer gives an all-center estimate with the generic power profile under normalization by $A_2(P)$ \cite{GlazerMikulincer22}.

For the block model \eqref{eq:unequal-product}, there are $M_{\mathbf m}^2$ top-degree monomials, each with coefficient $2^d/M_{\mathbf m}^2$. Consequently,
$$
A_\infty(P_{\mathbf m,d}) = \frac{2^d}{M_{\mathbf m}^2},\quad A_2(P_{\mathbf m,d}) = \frac{2^d}{M_{\mathbf m}}.
$$
Thus the maximum-coefficient theorem has the product-order profile but the scale $M_{\mathbf m}^2$, while the aggregate coefficient theorem has the natural scale $M_{\mathbf m}$ but the generic power profile. Theorem~\ref{thm:main-smallball} combines the product profile with the scale $M_{\mathbf m}$.

The normalizations are not uniformly comparable. For fixed $d$ and $n>d$, the polynomial
$$
P_n(x) = x_1\cdots x_d+n^{-2}\sum_{j = d+1}^nx_j
$$
has $A_\infty(P_n) = A_2(P_n) = 1$ and $\osc(P_n)\asymp1$. In every admissible partition, the variables $x_1,\ldots,x_d$ belong to distinct blocks; the terms in $w_d$ obtained by omitting one of these blocks show that $w_d\gtrsim_d\sqrt n$. The present theorem is therefore adapted to admissible block geometry rather than a universal replacement of coefficient normalization.

For homogeneous scalar block-multilinear forms, the tensor results of Hu and Paouris overlap the zero-centered product-profile phenomenon \cite{HuPaouris26}. The present results also cover nonhomogeneous multi-affine polynomials, are uniform over all centers, use global oscillation, allow arbitrary admissible partitions, and yield density and quotient-Remez estimates.

On the density side, published work of Kosov gives fractional Besov regularity under much broader hypotheses \cite{Kosov18,Kosov22}. In the Gaussian multi-affine setting, the modulus of continuity obtained in \cite{KosovIndividualDegree22} already implies that the density belongs to every finite $L^p$.\footnote{Indeed, $t\log^{d-1}(e/t)\lesssim_{d,\alpha}t^\alpha$ for every $\alpha<1$, and the standard one-dimensional Besov embedding $B^\alpha_{1,\infty}(\R)\hookrightarrow L^p(\R)$ applies whenever $\alpha>1-1/p$.} The new content of Theorem~\ref{thm:main-density} is the explicit and sharp $p^{d-1}$ norm growth with the global-oscillation/admissible-block normalization, together with matching unequal-block examples. The Remez application should be compared with the Brudnyi-Ganzburg theorem for arbitrary polynomials on convex bodies \cite{BrudnyiGanzburg73}.

The claims of sharpness in this paper have the following precise scope. They concern the small-argument order of $\Phi_d$, the dependence on the full block-size vector for prescribed admissible partitions with exactly $d$ blocks, and the $p^{d-1}$ density growth on the block-product models for $p\ge2$. No optimality is asserted for $w_d(\cB)$ when $\#\cB>d$, for the complete dependence on the number of blocks, for a different profile, or for arbitrary Remez observation sets.

For the singleton partition, Theorem~\ref{thm:main-smallball} gives the dimensional exponent $d-1/2$. The block-product models rule out every exponent below $d/2$ among uniform all-center oscillation-normalized estimates with profile $\Phi_d$. We do not know whether the exponent $d-1/2$ can be reduced; to the authors' knowledge, the optimal exponent in this setting is unknown.

\subsection{Proof strategy and organization}

The proof is based on one weighted contraction principle. A telescoping path between two points selects a block $B$ and a sign vector $w$ for which the scalar contraction $Q_{B,w}$ has large supremum norm. The weight assigned to $B$ is the recursive scale of the partition left after $B$ is removed. The pair $(B,w)$ is then fixed. Conditioning on the complementary variables and applying affine cube slicing in $B$ reduces the problem to the lower-degree polynomial $Q_{B,w}$. The child scale in the induction hypothesis cancels with the same scale in the selection inequality, and the integration closes through
$$
\int_0^1\Phi_{d-1}(s/t)dt = \Phi_d(s).
$$
The density proof uses the same fixed contraction and a negative-moment estimate. The Remez application follows by inverting the all-center small-ball bound.

Section~\ref{sec:consequences} records the main corollaries and examples. Section~\ref{sec:preliminaries} develops the product profile, affine slicing, block contractions, and the weighted recursive scale. Sections~\ref{sec:smallball}, \ref{sec:models}, and \ref{sec:density} prove the three main theorems. Section~\ref{sec:remez} proves the Remez application.

\section{Consequences, corollaries, and examples} \label{sec:consequences}

For a finite partition $\cB$, define
$$
\sigma(\cB) = \sum_{B\in\cB}\sqrt{\abs B},\quad w_*(P) = \min_{\cB\in\Adm(P)}w_{\deg(P)}(\cB).
$$

\begin{corollary}[Useful forms of the small-ball estimate] \label{cor:smallball-consequences}
Under the assumptions of Theorem~\ref{thm:main-smallball},
\begin{align} \label{eq:main-intrinsic}
\cQ_P(\rho)\le\Phi_d\left(C_dw_*(P)\frac{\rho}{\osc(P)}\right).
\end{align}
Moreover,
\begin{align} \label{eq:w-coarse-intro}
w_d(\cB)\le\sqrt n \sigma(\cB)^{d-1}.
\end{align}
Consequently,
\begin{align} \label{eq:main-sigma}
\cQ_P(\rho)\le\Phi_d\left(C_d\sqrt n \sigma(\cB)^{d-1}\frac{\rho}{\osc(P)}\right).
\end{align}
If $\cB$ has at most $q$ blocks, then
\begin{align} \label{eq:main-fixed-q}
\cQ_P(\rho)\le\Phi_d\left(C_dq^{(d-1)/2}n^{d/2}\frac{\rho}{\osc(P)}\right).
\end{align}
If $\cB = \{B_1,\ldots,B_d\}$ has exactly $d$ blocks, then
\begin{align} \label{eq:main-minimal-block}
\cQ_P(\rho)\le\Phi_d\left(C_d\prod_{j = 1}^d\sqrt{\abs{B_j}}\frac{\rho}{\osc(P)}\right).
\end{align}
\end{corollary}

\begin{proof}
Minimizing Theorem~\ref{thm:main-smallball} over $\Adm(P)$ gives \eqref{eq:main-intrinsic}. For every $(d-1)$-block family $\cI\subseteq\cB$,
$$
\sqrt{\sum_{C\in\cB\setminus\cI}\abs C}\le\sqrt n.
$$
The sum of the products $\prod_{B\in\cI}\sqrt{\abs B}$ is the $(d-1)$st elementary symmetric polynomial in the numbers $(\sqrt{\abs B})_{B\in\cB}$ and is bounded by $\sigma(\cB)^{d-1}$. This proves \eqref{eq:w-coarse-intro}. If $\#\cB\le q$, Cauchy-Schwarz gives
\begin{align} \label{eq:sigma-cs}
\sigma(\cB)^2\le qn.
\end{align}
This proves \eqref{eq:main-fixed-q}. Finally, when $\cB$ has exactly $d$ blocks, every $(d-1)$-block family is obtained by omitting one block, and
$$
w_d(\cB) = \sum_{j = 1}^d\left(\prod_{i\ne j}\sqrt{\abs{B_i}}\right)\sqrt{\abs{B_j}} = d\prod_{j = 1}^d\sqrt{\abs{B_j}}.
$$
Absorbing $d$ into $C_d$ proves \eqref{eq:main-minimal-block}.
\end{proof}

The weighted scale is explicit. If $\cB$ consists of $q\ge d$ equal blocks of size $n/q$, then
\begin{align} \label{eq:w-equal-blocks}
w_d(\cB) = \binom q{d-1}\sqrt{q-d+1}\left(\frac nq\right)^{d/2}\asymp_dq^{(d-1)/2}n^{d/2}.
\end{align}
For the singleton partition,
\begin{align} \label{eq:w-singleton}
w_d(\{\{1\},\ldots,\{n\}\}) = \binom n{d-1}\sqrt{n-d+1}\asymp_dn^{d-1/2}.
\end{align}
If $\cB$ has exactly $d$ blocks of sizes $M,1,\ldots,1$, then $w_d(\cB) = d\sqrt M$, whereas the coarse scale $\sqrt n \sigma(\cB)^{d-1}$ is of order $M^{d/2}$.

The monomial $M_d(x) = x_1\cdots x_d$ is the exact product prototype. Its distribution function is $\Phi_d$, and its density is
\begin{align} \label{eq:monomial-density-intro}
f_{M_d}(t) = \frac{\log^{d-1}(1/t)}{(d-1)!}\1_{(0,1)}(t).
\end{align}
In particular, $\norm{f_{M_d}}_{L^p(\R)}\asymp_dp^{d-1}$ as $p\to\infty$, and the density is unbounded when $d\ge2$.

\begin{proposition}[Sharp density scaling for block products] \label{prop:block-density}
Let $X\sim\mu_{N_{\mathbf m}}$, and let $f_{\mathbf m,d}$ be the density of $P_{\mathbf m,d}(X)$. Then
\begin{align} \label{eq:block-density}
\norm{f_{\mathbf m,d}}_{L^p(\R)}\asymp_dp^{d-1}M_{\mathbf m}^{1-1/p},\quad p\ge2,
\end{align}
uniformly over $\mathbf m\in\N^d$.
\end{proposition}

For a nonnegative measurable function $g$ on $\R$, let $g^{*}$ denote its right-continuous decreasing rearrangement,
$$
g^{*}(s) = \inf\{t\ge0:\vol_1\{y\in\R:g(y)>t\}\le s\},\quad s>0,
$$
where $\vol_1$ denotes Lebesgue measure on $\R$ and $\inf\varnothing = +\infty$.

\begin{corollary}[Logarithmic rearrangement] \label{cor:rearrangement}
Under the assumptions of Theorem~\ref{thm:main-density}, put
$$
K = \frac{w_d(\cB)}{\osc(P)},\quad g(y) = K^{-1}f_P(y/K).
$$
Then $g$ is the density of $KP(X)$ and
\begin{align} \label{eq:main-rearrangement}
g^{*}(s)\le C_d\log^{d-1}(e/s),\quad 0<s\le1.
\end{align}
The logarithmic power is optimal.
\end{corollary}

\begin{corollary}[Conventional Remez form] \label{cor:remez-sup}
Under the assumptions of Application~\ref{thm:main-remez},
\begin{align} \label{eq:main-remez-sup}
\norm P_{L^\infty([0,1]^n)}\le C_dw_d(\cB)\frac{\log^{d-1}(e/\theta)}{\theta}\norm P_{L^\infty(\Omega,\mu_n)}.
\end{align}
\end{corollary}

For comparison, the Brudnyi-Ganzburg inequality for arbitrary degree-$d$ polynomials on convex bodies has, for fixed $d$ and small relative measure $\theta$, a factor of order $(n/\theta)^d$ \cite{BrudnyiGanzburg73,BrudnyiYomdin16}. Under an admissible partition with at most $q$ blocks, Application~\ref{thm:main-remez} gives instead
$$
C_dq^{(d-1)/2}n^{d/2}\frac{\log^{d-1}(e/\theta)}{\theta}.
$$
The improvement reflects the additional admissible block structure.

For $s>0$ and $X\sim\mu_{N_{\mathbf m}}$, put
$$
\Omega_{\mathbf m,s} = \left\{x:\abs{P_{\mathbf m,d}(x)}\le\frac{s}{M_{\mathbf m}}\right\},\quad \theta_{\mathbf m,s} = \Prob\{X\in\Omega_{\mathbf m,s}\}.
$$

\begin{proposition}[Remez lower model] \label{prop:remez-model}
For every $d\ge1$, there are constants $a_d,b_d,\kappa_d,s_d>0$ such that, for every $\mathbf m\in\N^d$ and $0<s\le s_d$,
\begin{align} \label{eq:theta-two-sided}
\Phi_d(a_ds)\le\theta_{\mathbf m,s}\le\Phi_d(b_ds)
\end{align}
and
\begin{align} \label{eq:remez-model}
\frac{\osc(P_{\mathbf m,d})}{r_{\Omega_{\mathbf m,s}}(P_{\mathbf m,d})}\ge\kappa_d\frac{M_{\mathbf m}}{\Phi_d^{-1}(\theta_{\mathbf m,s})}.
\end{align}
\end{proposition}

\section{Preliminaries} \label{sec:preliminaries}

For a finite partition $\cC$, put $N(\cC) = \sum_{C\in\cC}\abs C$ and $N(\varnothing) = 0$. If $E$ is a finite coordinate set and $I\subseteq E$, write $I^c = E\setminus I$, with the ambient set always clear from the context. We identify $[0,1]^I$ and $\R^I$ with the spaces of vectors indexed by $I$. For $x\in[0,1]^E$, write $x_I = (x_i)_{i\in I}$. Let $\mu_I$ denote product probability measure on $[0,1]^I$; in particular, $\mu_\varnothing$ is unit mass on the one-point space $[0,1]^\varnothing$. If $I = [m]$, we also write $\mu_m$. When a function on $[0,1]^I$ depends only on coordinates in $A\subseteq I$, we identify it with the corresponding function on $[0,1]^A$; integrating the coordinates in $I\setminus A$ then contributes the factor one. For vectors, $\norm{\cdot}_p$ denotes the usual $\ell^p$-norm, while function norms are written as $\norm{\cdot}_{L^p}$; in particular, $L^\infty$-norms are essential suprema. We write $\vol_k$ for $k$-dimensional Lebesgue measure. The notions of active variables, admissible partitions, and block cost are used verbatim for polynomials indexed by an arbitrary finite coordinate set. Finally, $\norm P_\infty$ denotes the supremum norm on the full cube.

\subsection{Elementary facts and the product profile}
\begin{lemma}[Vertex extrema] \label{lem:vertex}
A multi-affine polynomial attains its maximum and minimum on $\{0,1\}^n$.
\end{lemma}

\begin{proof}
With all other coordinates fixed, the polynomial is affine in $x_n$, so an extremum in that coordinate occurs at an endpoint. Repeating this operation for $x_{n-1},\ldots,x_1$ moves an extremizer to a vertex without worsening its value.
\end{proof}

\begin{lemma}[Null level sets] \label{lem:zero-set}
If $P$ is a nonzero multi-affine polynomial, then $\mu_n\{P = 0\} = 0$. Consequently, every nonconstant multi-affine polynomial satisfies
\begin{align} \label{eq:Q-zero}
\cQ_P(0) = 0.
\end{align}
\end{lemma}

\begin{proof}
Induct on $n$. If $n = 1$, a nonzero affine function is either a nonzero constant, whose zero set is empty, or has at most one zero. Thus the assertion holds in dimension one. Write in dimension $n$
$$
P(x',x_n) = x_nQ(x')+R(x'),
$$
where $Q$ and $R$ are not both zero. For fixed $x'$, the zero fiber has one-dimensional measure zero unless $Q(x') = R(x') = 0$. Fubini therefore gives
$$
\mu_n\{P = 0\} = \mu_{n-1}(\{Q = 0\}\cap\{R = 0\}).
$$
Every nonzero one of $Q,R$ has a null zero set by induction, so the right-hand side is zero. Applying the result to $P-u$ proves \eqref{eq:Q-zero} for every $u$.
\end{proof}

Consequently, for a nonconstant multi-affine polynomial, replacing a strict sublevel inequality by the corresponding non-strict inequality does not change its measure. We use this fact without further comment.

\begin{lemma}[Profile calculus] \label{lem:Phi-calculus}
For every $d\ge1$, $\Phi_d$ is continuous on $[0,\infty)$ and strictly increasing on $[0,1]$. For $0<t<1$,
\begin{align} \label{eq:Phi-derivative}
\Phi_d'(t) = \frac{\log^{d-1}(1/t)}{(d-1)!}.
\end{align}
If $1\le r\le d$, then $\Phi_r\le\Phi_d$. For $d\ge2$ and $s\ge0$,
\begin{align} \label{eq:Phi-recursion}
\int_0^1\Phi_{d-1}(s/t) dt = \Phi_d(s).
\end{align}
Finally, for $a,b>0$,
\begin{align} \label{eq:Phi-ratio}
\lim_{t\downarrow0}\frac{\Phi_d(at)}{\Phi_d(bt)} = \frac ab.
\end{align}
\end{lemma}

\begin{proof}
For $j\ge1$,
$$
\frac{d}{dt}\left(\frac{t\log^j(1/t)}{j!}\right) = \frac{\log^j(1/t)}{j!}-\frac{\log^{j-1}(1/t)}{(j-1)!}.
$$
Adding these identities to the derivative of the term $t$ makes all intermediate terms cancel and proves \eqref{eq:Phi-derivative}. Continuity at zero follows from $t\log^k(1/t)\to0$. The inequality $\Phi_r\le\Phi_d$ for $r\le d$ follows because the defining sum for $\Phi_d$ contains all terms of the sum for $\Phi_r$ and additional nonnegative terms.

If $s = 0$, both sides of \eqref{eq:Phi-recursion} are zero. If $s\ge1$, then $s/t\ge1$ for $0<t\le1$, so the integrand equals one and both sides equal one. It remains to consider $0<s<1$, for which
\begin{align*}
\int_0^1\Phi_{d-1}(s/t) dt = s+\int_s^1\frac{s}{t} \sum_{j = 0}^{d-2}\frac{\log^j(t/s)}{j!} dt = s+s\sum_{j = 0}^{d-2}\frac1{j!} \int_0^{\log(1/s)}u^j du = \Phi_d(s).
\end{align*}
The asymptotic \eqref{eq:Phi-asymp-intro} gives \eqref{eq:Phi-ratio}.
\end{proof}

\begin{lemma}[Products of uniforms] \label{lem:Phi-product}
Let $U_1,\ldots,U_d$ be independent and uniform on $[0,1]$. Then
\begin{align} \label{eq:uniform-product}
\Prob\{U_1\cdots U_d\le t\} = \Phi_d(t), \quad t\ge0.
\end{align}
On $(0,1)$ the product has density
\begin{align} \label{eq:uniform-product-density}
h_d(t) = \frac{\log^{d-1}(1/t)}{(d-1)!}.
\end{align}
\end{lemma}

\begin{proof}
Since $U_j>0$ almost surely, the variables $E_j = -\log U_j$ are well defined almost surely and are independent exponentials of rate one. Their sum $G_d = E_1+\cdots+E_d$ has gamma density $e^{-u}u^{d-1}/(d-1)!$, and repeated integration by parts gives
$$
\Prob\{G_d\ge u\} = e^{-u}\sum_{j = 0}^{d-1}\frac{u^j}{j!}, \quad u\ge0.
$$
Since $-\log(U_1\cdots U_d) = G_d$, this proves \eqref{eq:uniform-product}. Differentiation on $(0,1)$ gives \eqref{eq:uniform-product-density}.
\end{proof}

\begin{lemma}[Layer cake] \label{lem:layer-cake}
If $g\ge0$ is measurable on a probability space $(S,\nu)$, then
\begin{align} \label{eq:layer-cake}
\int_S\min\{1,g\} d\nu = \int_0^1\nu\{g>t\} dt.
\end{align}
\end{lemma}

\begin{proof}
Use $\min\{1,g(x)\} = \int_0^1\1_{\{g(x)>t\}} dt$ and Tonelli's theorem.
\end{proof}

\begin{lemma}[Profile domination and inverse moments] \label{lem:inverse-moments}
Let $Z>0$ almost surely and suppose
$$
\Prob\{Z\le t\}\le\Phi_d(t),\quad t\ge0.
$$
Then, for $0<\eta<1$,
\begin{align} \label{eq:inverse-moment}
\E Z^{-\eta}\le(1-\eta)^{-d}.
\end{align}
\end{lemma}

\begin{proof}
Let $U = U_1\cdots U_d$ as in Lemma~\ref{lem:Phi-product}. For a distribution function $F$, let
$$
F^{-1}(v) = \inf\{t:F(t)\ge v\},\quad 0<v<1.
$$
The distribution-function inequality implies $F_Z^{-1}(v)\ge F_U^{-1}(v)$ for every $v\in(0,1)$. Coupling both generalized inverses with one uniform random variable and using that $x\mapsto x^{-\eta}$ is decreasing gives
$$
\E Z^{-\eta}\le\E U^{-\eta} = \prod_{j = 1}^d\int_0^1u^{-\eta} du = (1-\eta)^{-d}.
$$
This argument allows atoms; positivity rules out an atom at zero.
\end{proof}

\begin{lemma}[Inverse product profile] \label{lem:inverse-Phi}
For every $d\ge1$, there are constants $0<c_d\le C_d<\infty$ such that
\begin{align} \label{eq:inverse-Phi}
c_d\frac{\theta}{\log^{d-1}(e/\theta)} \le\Phi_d^{-1}(\theta) \le C_d\frac{\theta}{\log^{d-1}(e/\theta)}, \quad 0<\theta\le1.
\end{align}
\end{lemma}

\begin{proof}
Put $L = \log(e/\theta)\ge1$. We first prove the lower bound. Let $t = c\theta/L^{d-1}$ with $0<c\le1$. Since $\theta = e^{1-L}$,
$$
\log(1/t) = L-1+(d-1)\log L+\log(1/c).
$$
For $L\ge1$, the elementary bound $\log L\le L$ gives $\log(1/t)\le(d+\log(1/c))L$. Hence
$$
\Phi_d(t)\le\frac{c\theta}{L^{d-1}}\sum_{j = 0}^{d-1}\frac{((d+\log(1/c))L)^j}{j!}\le cC_d(1+\log(1/c))^{d-1}\theta.
$$
Since $c(1+\log(1/c))^{d-1}\to0$ as $c\downarrow0$, choosing $c = c_d>0$ sufficiently small makes the last expression at most $\theta$. Since $\Phi_d$ is increasing, $t\le\Phi_d^{-1}(\theta)$.

For the upper bound, let $t = A_d\theta/L^{d-1}$. Choose $L_d$ so large that, whenever $L\ge L_d$, one has $t<1$ and
$$
\log(1/t) = L-1+(d-1)\log L-\log A_d\ge L/2
$$
for a fixed sufficiently large constant $A_d$. Then
$$
\Phi_d(t)\ge\frac{t\log^{d-1}(1/t)}{(d-1)!}\ge\frac{A_d\theta}{2^{d-1}(d-1)!}.
$$
Taking $A_d\ge2^{d-1}(d-1)!$ gives $\Phi_d(t)\ge\theta$, so $\Phi_d^{-1}(\theta)\le t$ whenever $L\ge L_d$. On the remaining compact interval $e^{1-L_d}\le\theta\le1$, the continuous function $\Phi_d^{-1}(\theta)L^{d-1}/\theta$ is bounded. Increasing $C_d$ completes the proof.
\end{proof}

\subsection{Affine slicing of the cube}

The geometric input is Ball's cube-slicing theorem \cite{Ball86}.
\begin{proposition}[Affine slicing] \label{prop:affine-slicing}
Let $a\in\R^m\setminus\{0\}$, $b\in\R$, and let $Z$ be uniform on $[0,1]^m$. The law of $a\cdot Z+b$ has a density $h$ satisfying
\begin{align} \label{eq:affine-Linfty}
\norm h_{L^\infty(\R)}\le\frac{\sqrt2}{\norm a_2}.
\end{align}
Consequently, for $u\in\R$, $\rho\ge0$, and $1<p<\infty$,
\begin{align}
\Prob\{\abs{a\cdot Z+b-u}\le\rho\}
&\le\min\left\{1,\frac{2\sqrt2\rho}{\norm a_2}\right\},
 \label{eq:affine-smallball} \\
\norm h_{L^p(\R)}
&\le\left(\frac{\sqrt2}{\norm a_2}\right)^{1-1/p}.
 \label{eq:affine-Lp}
\end{align}
\end{proposition}

\begin{proof}
For $m = 1$, the random variable $a_1Z_1+b$ is uniform on an interval of length $\abs{a_1}$, so the assertions follow directly. Assume $m\ge2$.

Let $L(x) = a\cdot x+b$, and let $\mathcal H^{m-1}$ denote $(m-1)$-dimensional Hausdorff measure. The coarea formula applied to $L$ gives, for every nonnegative Borel function $\varphi$,
$$
\int_{[0,1]^m}\varphi(L(x))dx = \int_\R\varphi(t)\frac{\mathcal H^{m-1}([0,1]^m\cap\{x:L(x) = t\})}{\norm a_2}dt.
$$
Thus the law of $L(Z)$ has the density
$$
h(t) = \frac{\mathcal H^{m-1}([0,1]^m\cap\{x:a\cdot x+b = t\})}{\norm a_2}
$$
for almost every $t$.

Translate $[0,1]^m$ to the centered cube $K = [-1/2,1/2]^m$, and let $A(s)$ be the $(m-1)$-dimensional volume of the section $K\cap\{x:a\cdot x = s\}$. By Brunn's concavity theorem \cite{Gardner02}, $A(s)^{1/(m-1)}$ is concave on its support. Since $K$ is centrally symmetric, $A$ is even, and hence $A(s)\le A(0)$ for every $s$. Ball's cube-slicing theorem \cite{Ball86} gives $A(0)\le\sqrt2$. Consequently,
$$
\norm{h}_{L^\infty(\R)}\le\frac{\sqrt2}{\norm a_2},
$$
which is \eqref{eq:affine-Linfty}.

The event $\abs{L(Z)-u}\le\rho$ is the inverse image of an interval of length $2\rho$. Therefore
$$
\Prob\{\abs{L(Z)-u}\le\rho\}\le2\rho\norm{h}_{L^\infty(\R)},
$$
and the probability is also at most one. This proves \eqref{eq:affine-smallball}. Finally, since $h$ is a probability density,
$$
\norm{h}_{L^p(\R)}^p = \int_\R h(t)^pdt\le\norm{h}_{L^\infty(\R)}^{p-1}\int_\R h(t)dt = \norm{h}_{L^\infty(\R)}^{p-1}.
$$
Taking $p$th roots gives \eqref{eq:affine-Lp}.
\end{proof}

The density proof requires a jointly measurable choice of the conditional affine densities.

\begin{lemma}[Measurable affine disintegration] \label{lem:disintegration}
Let $(\mathsf Y,\mathcal A,\nu)$ be a probability space, let $R:\mathsf Y\to\R$ and $a:\mathsf Y\to\R^b$ be measurable, and assume $a(y)\ne0$ for $\nu$-almost every $y$. If $Y\sim\nu$, $Z\sim\operatorname{Unif}[0,1]^b$ is independent of $Y$, and
$$
T = R(Y)+a(Y)\cdot Z,
$$
then there is a jointly measurable $h:\mathsf Y\times\R\to[0,\infty)$ such that $h_y = h(y,\cdot)$ is a density of $R(y)+a(y)\cdot Z$ for almost every $y$ and
\begin{align} \label{eq:mixture-density}
f_T(t) = \int_{\mathsf Y}h_y(t) d\nu(y)
\end{align}
is a density of $T$. Moreover, if $1\le p<\infty$ and
$$
\int_{\mathsf Y}\norm{h_y}_{L^p(\R)}d\nu(y)<\infty,
$$
then $f_T\in L^p(\R)$ and
\begin{align} \label{eq:mixture-Minkowski}
\norm{f_T}_{L^p(\R)}
\le
\int_{\mathsf Y}\norm{h_y}_{L^p(\R)}d\nu(y).
\end{align}
\end{lemma}

\begin{proof}[Proof of Lemma~\ref{lem:disintegration}]
Put $M(y) = \norm{a(y)}_\infty$. For $1\le j\le b$, define
$$
D_j = \{y:a(y)\ne0, \abs{a_j(y)} = M(y), \abs{a_i(y)}<M(y)\text{ for }i<j\}.
$$
The sets $D_j$ are measurable, pairwise disjoint, and their union is $\{a\ne0\}$. Thus the smallest index at which the maximum coefficient is attained is selected measurably.

Fix $j$ and put $J_j = [b]\setminus\{j\}$. For $y\in D_j$ and $z_{J_j}\in[0,1]^{J_j}$, define
$$
c_j(y,z_{J_j}) = R(y)+\sum_{i\in J_j}a_i(y)z_i.
$$
Let $I_j(y,z_{J_j})$ be the closed interval with endpoints $c_j(y,z_{J_j})$ and $c_j(y,z_{J_j})+a_j(y)$. Define $h(y,t) = 0$ when $a(y) = 0$, and when $a(y)\ne0$ set
\begin{align} \label{eq:explicit-h}
h(y,t) = \sum_{j = 1}^b\1_{D_j}(y)\int_{[0,1]^{J_j}}\frac{\1_{I_j(y,z_{J_j})}(t)}{\abs{a_j(y)}}d\mu_{J_j}(z_{J_j}).
\end{align}
The endpoints of $I_j(y,z_{J_j})$ are measurable functions of $(y,z_{J_j})$. Hence the integrand in \eqref{eq:explicit-h} is jointly measurable in $(y,z_{J_j},t)$, and parameter integration shows that $h$ is jointly measurable in $(y,t)$.

Fix $y\in D_j$. Conditional on $Z_{J_j} = z_{J_j}$,
$$
R(y)+a(y)\cdot Z = c_j(y,z_{J_j})+a_j(y)Z_j.
$$
Since $a_j(y)\ne0$, the last random variable is uniform on $I_j(y,z_{J_j})$ and has density
$$
t\mapsto\frac{\1_{I_j(y,z_{J_j})}(t)}{\abs{a_j(y)}}.
$$
Averaging this conditional density over $z_{J_j}$ proves that $h_y = h(y,\cdot)$ is a density of $R(y)+a(y)\cdot Z$.

For every nonnegative Borel function $\varphi$, independence and Tonelli's theorem give
$$
\E\varphi(T) = \int_{\mathsf Y}\int_\R\varphi(t)h(y,t)dtd\nu(y) = \int_\R\varphi(t)\left(\int_{\mathsf Y}h(y,t)d\nu(y)\right)dt.
$$
Therefore $f_T(t) = \int_{\mathsf Y}h(y,t)d\nu(y)$ is a density of $T$, proving \eqref{eq:mixture-density}. Finally, under the stated finiteness assumption, Minkowski's integral inequality gives
$$
\norm{f_T}_{L^p(\R)} = \norm{\int_{\mathsf Y}h_y d\nu(y)}_{L^p(\R)}\le\int_{\mathsf Y}\norm{h_y}_{L^p(\R)}d\nu(y).
$$
This is \eqref{eq:mixture-Minkowski}.
\end{proof}

\begin{proposition}[Affine base case] \label{prop:affine-base}
Let $L$ be a nonconstant affine polynomial on $[0,1]^n$, let $X\sim\mu_n$, and let $f_L$ denote the density of $L(X)$. Then
\begin{align} \label{eq:affine-base-smallball}
\mu_n\{\abs L\le\rho\} \le\Phi_1 \left(4\sqrt{2n}\frac{\rho}{\norm L_\infty}\right), \quad \rho\ge0,
\end{align}
and its density satisfies
\begin{align} \label{eq:affine-base-density}
\norm{f_L}_{L^p(\R)} \le\left(\frac{\sqrt{2n}}{\osc(L)}\right)^{1-1/p}, \quad 1<p<\infty.
\end{align}
\end{proposition}

\begin{proof}
Write $L(x) = a\cdot x+b$. Since $L$ is nonconstant, $a\ne0$. We first prove \eqref{eq:affine-base-smallball} and normalize $\norm L_\infty = 1$. If $4\sqrt{2n}\rho\ge1$, then the right-hand side equals one and the estimate is trivial. Thus we may assume $\rho<1/(4\sqrt{2n})\le1/2$. If the sublevel set is empty, there is nothing to prove. Otherwise the interval $L([0,1]^n)$ meets $[-\rho,\rho]$. By Lemma~\ref{lem:vertex}, one endpoint of this interval has absolute value one. Therefore its length is at least $1-\rho$, and hence
$$
\osc(L)\ge1-\rho\ge\frac12.
$$
For an affine function on the cube, the range length is
$$
\osc(L) = \sum_{i = 1}^n\abs{a_i} = \norm a_1.
$$
Cauchy-Schwarz now gives $\norm a_2\ge\norm a_1/\sqrt n\ge1/(2\sqrt n)$. Proposition~\ref{prop:affine-slicing} yields
$$
\mu_n\{\abs L\le\rho\}\le\frac{2\sqrt2\rho}{\norm a_2}\le4\sqrt{2n}\rho = \Phi_1(4\sqrt{2n}\rho),
$$
which proves \eqref{eq:affine-base-smallball} under the normalization. Scaling gives the general case.

For the density estimate, no normalization is needed. The identity $\osc(L) = \norm a_1$ and Cauchy-Schwarz give $\norm a_2\ge\osc(L)/\sqrt n$. Substitution into \eqref{eq:affine-Lp} proves \eqref{eq:affine-base-density}.
\end{proof}

\subsection{Block contractions}

Let $P$ be a multi-affine polynomial in the variables indexed by $[n]$, let $\cB\in\Adm(P)$, and fix $B\in\cB$. Write $B^c = [n]\setminus B$. Since no monomial of $P$ contains two variables from $B$, there are unique multi-affine polynomials
$$
R_B:[0,1]^{B^c}\to\R,\quad Q_{B,i}:[0,1]^{B^c}\to\R\quad(i\in B)
$$
such that
\begin{align} \label{eq:block-representation}
P(x) = R_B(x_{B^c})+\sum_{i\in B}x_iQ_{B,i}(x_{B^c}).
\end{align}
If $y = x_{B^c}$ and $z = x_B$, then
$$
P(y,z) = R_B(y)+a_B(y)\cdot z,\quad a_B(y) = (Q_{B,i}(y))_{i\in B}\in\R^B.
$$
For $w = (w_i)_{i\in B}\in\R^B$, define the scalar $B$-contraction
\begin{align} \label{eq:contraction}
Q_{B,w}(y) = a_B(y)\cdot w = \sum_{i\in B}w_iQ_{B,i}(y),\quad y\in[0,1]^{B^c}.
\end{align}

We shall compare partitions that may live on different finite ground sets. For two such partitions $\cA$ and $\cC$, write $\cA\preccurlyeq\cC$ if there is an injection $\iota:\cA\to\cC$ such that $A\subseteq\iota(A)$ for every $A\in\cA$. Thus $\cA\preccurlyeq\cC$ means that $\cA$ is obtained from a subfamily of $\cC$ by deleting variables from its blocks and discarding the blocks that become empty.

\begin{lemma}[Inherited partition] \label{lem:inherited}
For every $i\in B$ and $w\in\R^B$, the polynomials $Q_{B,i}$ and $Q_{B,w}$ are multi-affine and have degree at most $\deg(P)-1$. Assume that $Q_{B,w}$ is nonconstant, and let $A_{B,w}\subseteq B^c$ be its active-variable set. Then
\begin{align} \label{eq:inherited-family}
\cB_{B,w} = \{C\cap A_{B,w}:C\in\cB\setminus\{B\}, C\cap A_{B,w}\ne\varnothing\}
\end{align}
is an admissible partition of $A_{B,w}$ for $Q_{B,w}$. Moreover,
\begin{align} \label{eq:inherited-order}
\cB_{B,w}\preccurlyeq\cB\setminus\{B\},\quad \sigma(\cB_{B,w})\le\sigma(\cB)-\sqrt{\abs B}.
\end{align}
\end{lemma}

\begin{proof}
Write $P$ in the square-free expansion \eqref{eq:expansion}. By admissibility, the support of each nonzero monomial either avoids $B$ or meets $B$ in a unique index $i\in B$. Grouping the monomials according to these alternatives gives \eqref{eq:block-representation}; uniqueness follows from uniqueness of the square-free monomial expansion.

Every monomial contributing to $Q_{B,i}$ is obtained from a monomial of $P$ containing $x_i$ by removing the factor $x_i$. Thus $Q_{B,i}$ is multi-affine and $\deg(Q_{B,i})\le\deg(P)-1$. The same assertions hold for $Q_{B,w}$ because it is a linear combination of the $Q_{B,i}$.

A linear combination may cancel monomials but cannot create a new monomial support. Hence every support occurring in $Q_{B,w}$ meets each block $C\in\cB\setminus\{B\}$ in at most one variable. After the inactive variables of $Q_{B,w}$ are discarded, the nonempty intersections $C\cap A_{B,w}$ form the partition \eqref{eq:inherited-family}, which is admissible for $Q_{B,w}$. The inclusion $C\cap A_{B,w}\subseteq C$ gives \eqref{eq:inherited-order}. Finally,
$$
\sigma(\cB_{B,w}) = \sum_{{C\in\cB\setminus\{B\}, C\cap A_{B,w}\ne\varnothing}}\sqrt{\abs{C\cap A_{B,w}}}\le\sum_{C\in\cB\setminus\{B\}}\sqrt{\abs C} = \sigma(\cB)-\sqrt{\abs B}.
$$
\end{proof}

\begin{lemma}[Weighted block slope principle] \label{lem:block-slope}
Let $P$ be a nonconstant multi-affine polynomial, let $\cB\in\Adm(P)$, let $x,z\in[0,1]^n$, and let $(c_B)_{B\in\cB}$ be positive numbers. Then there are $B\in\cB$ and $w\in\{-1,1\}^B$ such that
\begin{align} \label{eq:block-slope}
\norm{Q_{B,w}}_\infty\ge\frac{c_B}{\sum_{C\in\cB}c_C}\abs{P(z)-P(x)}.
\end{align}
\end{lemma}

\begin{proof}
Order the blocks as $B_1,\ldots,B_q$. For $0\le j\le q$, let $z^{(j)}\in[0,1]^n$ be the point whose coordinates agree with $z$ on $B_1\cup\cdots\cup B_j$ and with $x$ on the remaining blocks. Thus $z^{(0)} = x$ and $z^{(q)} = z$. Put
$$
\Delta_j = P(z^{(j)})-P(z^{(j-1)}),\quad 1\le j\le q.
$$
During the passage from $z^{(j-1)}$ to $z^{(j)}$, only the coordinates in $B_j$ change. Let $\xi_j\in[0,1]^{B_j^c}$ be their common complementary coordinate vector. The block representation gives
$$
\Delta_j = a_{B_j}(\xi_j)\cdot(z_{B_j}-x_{B_j}).
$$
Since every coordinate of $z_{B_j}-x_{B_j}$ has absolute value at most one,
\begin{align} \label{eq:increment-l1}
\abs{\Delta_j}\le\norm{a_{B_j}(\xi_j)}_1.
\end{align}
The telescoping identity and the triangle inequality give
$$
\abs{P(z)-P(x)}\le\sum_{j = 1}^q\abs{\Delta_j}.
$$
If $P(z) = P(x)$, the conclusion is trivial. Otherwise, there is an index $j$ such that
$$
\abs{\Delta_j}\ge\frac{c_{B_j}}{\sum_{C\in\cB}c_C}\abs{P(z)-P(x)};
$$
indeed, the strict reverse inequality for every $j$ would contradict the preceding estimate. Fix such a block $B = B_j$. Choose $w\in\{-1,1\}^B$ so that $w_iQ_{B,i}(\xi_j) = \abs{Q_{B,i}(\xi_j)}$ for every $i\in B$, choosing either sign when the coefficient is zero. Then
$$
Q_{B,w}(\xi_j) = \sum_{i\in B}\abs{Q_{B,i}(\xi_j)} = \norm{a_B(\xi_j)}_1\ge\abs{\Delta_j},
$$
where the last inequality is \eqref{eq:increment-l1}. Taking the supremum of $\abs{Q_{B,w}}$ over $[0,1]^{B^c}$ proves \eqref{eq:block-slope}.
\end{proof}

The final observation connects the scalar sign contraction to the Euclidean slope required by cube slicing.

\begin{lemma}[Block fiber estimates] \label{lem:block-fiber}
Let $J$ be a finite coordinate set, let
$$
P(y,z) = R(y)+a(y)\cdot z,\quad y\in[0,1]^J,\quad z\in[0,1]^b,
$$
and let $Z\sim\mu_b$. Fix $w\in\{-1,1\}^b$ and put $Q(y) = a(y)\cdot w$. For fixed $y$, every $u\in\R$, and $\rho\ge0$,
\begin{align} \label{eq:block-fiber-smallball}
\Prob_Z\{\abs{P(y,Z)-u}\le\rho\}\le\min\left\{1,\frac{2\sqrt{2b}\rho}{\abs{Q(y)}}\right\},
\end{align}
where $\Prob_Z$ denotes probability with respect to $Z$ and the quotient is $+\infty$ if $Q(y) = 0$. If $a(y)\ne0$ and $h_y$ is the density of $R(y)+a(y)\cdot Z$, then
\begin{align} \label{eq:block-fiber-Lp}
\norm{h_y}_{L^p(\R)}\le\left(\frac{\sqrt{2b}}{\abs{Q(y)}}\right)^{1-1/p}.
\end{align}
In \eqref{eq:block-fiber-Lp}, the right-hand side is interpreted as $+\infty$ when $Q(y) = 0$.
\end{lemma}

\begin{proof}
Since $w\in\{-1,1\}^b$, one has $\norm w_2 = \sqrt b$. Cauchy-Schwarz therefore gives
$$
\abs{Q(y)} = \abs{a(y)\cdot w}\le\sqrt b\norm{a(y)}_2.
$$
If $Q(y)\ne0$, then $a(y)\ne0$ and hence $\norm{a(y)}_2\ge\abs{Q(y)}/\sqrt b$. Applying \eqref{eq:affine-smallball} to the affine map $z\mapsto R(y)+a(y)\cdot z$ gives
$$
\Prob_Z\{\abs{P(y,Z)-u}\le\rho\}\le\min\left\{1,\frac{2\sqrt2\rho}{\norm{a(y)}_2}\right\}\le\min\left\{1,\frac{2\sqrt{2b}\rho}{\abs{Q(y)}}\right\}.
$$
When $Q(y) = 0$, the right-hand side is interpreted as one, so the same estimate remains valid. If $a(y)\ne0$, \eqref{eq:affine-Lp} and the same lower bound for $\norm{a(y)}_2$ give \eqref{eq:block-fiber-Lp}.
\end{proof}

\subsection{The weighted recursive scale}

For a finite partition $\cC$ with $q = \#\cC$, and for $0\le j\le q-1$, put
\begin{align} \label{eq:Wj}
W_j(\cC) = \sum_{{\cI\subseteq\cC, \abs{\cI} = j}}\left(\prod_{B\in\cI}\sqrt{\abs B}\right)\sqrt{N(\cC\setminus\cI)}.
\end{align}
Thus $W_j(\cC) = w_{j+1}(\cC)$ and $W_0(\cC) = \sqrt{N(\cC)}$.

\begin{lemma}[Weighted block identities] \label{lem:weighted-identities}
Let $\cC$ be a partition into $q\ge1$ nonempty blocks.
\begin{enumerate}
\item For $0\le j\le q-1$, one has $W_j(\cC)\ge1$.
\item For $1\le j\le q-1$,
\begin{align} \label{eq:W-monotone}
W_{j-1}(\cC)\le jW_j(\cC).
\end{align}
\item For $1\le j\le q-1$,
\begin{align} \label{eq:W-deletion}
\sum_{B\in\cC}\sqrt{\abs B} W_{j-1}(\cC\setminus\{B\}) = jW_j(\cC).
\end{align}
\item For $0\le j\le q-1$,
\begin{align} \label{eq:W-coarse}
W_j(\cC)\le\sqrt{N(\cC)} \sigma(\cC)^j.
\end{align}
\end{enumerate}
\end{lemma}

\begin{proof}
Write $x_B = \sqrt{\abs B}$ for $B\in\cC$. If $\abs{\cI} = j\le q-1$, then $\cC\setminus\cI$ is nonempty. Since every block has at least one element,
$$
\left(\prod_{B\in\cI}x_B\right)\sqrt{N(\cC\setminus\cI)}\ge1.
$$
At least one such family $\cI$ exists, proving the first assertion.

Fix $1\le j\le q-1$. For $\cI\subseteq\cC$ with $\abs{\cI} = j-1$, put $R_{\cI} = \sqrt{N(\cC\setminus\cI)}$. The complement $\cC\setminus\cI$ contains at least two blocks. Hence, for every $C\in\cC\setminus\cI$, the partition $\cC\setminus(\cI\cup\{C\})$ is nonempty and
$$
\sqrt{N(\cC\setminus(\cI\cup\{C\}))}\ge1.
$$
The inequality between the Euclidean and $\ell^1$ norms gives
$$
R_{\cI}\le\sum_{C\in\cC\setminus\cI}x_C\le\sum_{C\in\cC\setminus\cI}x_C\sqrt{N(\cC\setminus(\cI\cup\{C\}))}.
$$
Multiply by $\prod_{B\in\cI}x_B$ and sum over all $\cI$ of size $j-1$. Every family $\cJ\subseteq\cC$ of size $j$ occurs exactly $j$ times, once for each choice of $C\in\cJ$ with $\cI = \cJ\setminus\{C\}$. This proves \eqref{eq:W-monotone}.

For \eqref{eq:W-deletion}, expand the left-hand side. A term is indexed by a pair $(B,\cI)$ with $B\in\cC$ and $\cI\subseteq\cC\setminus\{B\}$ of size $j-1$. Put $\cJ = \cI\cup\{B\}$. The corresponding summand equals
$$
\left(\prod_{C\in\cJ}\sqrt{\abs C}\right)\sqrt{N(\cC\setminus\cJ)}.
$$
Every $\cJ$ of size $j$ has exactly $j$ representations of this form, proving the identity.

Finally, for every $\cI$ one has $N(\cC\setminus\cI)\le N(\cC)$. Therefore
$$
W_j(\cC)\le\sqrt{N(\cC)}\sum_{{\cI\subseteq\cC, \abs{\cI} = j}}\prod_{B\in\cI}x_B.
$$
The last sum is the $j$th elementary symmetric polynomial in the nonnegative numbers $(x_B)_{B\in\cC}$, and it is at most $(\sum_Bx_B)^j/j!\le\sigma(\cC)^j$. This proves \eqref{eq:W-coarse}.
\end{proof}

We now define the weighted recursive scale used in the induction. For a finite partition $\cC$, set
\begin{align}
 a_1(\cC)& = \max\{1,4\sqrt{2N(\cC)}\}, \label{eq:recursive-a1} \\
 a_k(\cC)& = \max\left\{ a_{k-1}(\cC),1+2\sqrt2\sum_{B\in\cC}\sqrt{\abs B} a_{k-1}(\cC\setminus\{B\})\right\},\quad k\ge2. \label{eq:recursive-ak}
\end{align}
The sum over the empty partition is zero, so $ a_k(\varnothing) = 1$. The first term makes the scale nondecreasing in the degree parameter and covers contractions whose degree drops by more than one. The weighted sum records the scale of each possible residual partition separately.

\begin{lemma}[Monotonicity and explicit partition scale] \label{lem:recursive-scale}
The scales in \eqref{eq:recursive-a1}-\eqref{eq:recursive-ak} have the following properties.
\begin{enumerate}
\item $ a_k(\cC)$ is nondecreasing in $k$.
\item If $\cA\preccurlyeq\cC$, then
\begin{align} \label{eq:a-restriction}
 a_k(\cA)\le a_k(\cC),\quad k\ge1.
\end{align}
\item If $\#\cC\ge k$, then
\begin{align} \label{eq:a-weighted}
 a_k(\cC)\le C_kW_{k-1}(\cC) = C_k w_k(\cC).
\end{align}
Consequently, if $N(\cC) = n\ge1$ and $\#\cC\ge k$, then
\begin{align} \label{eq:a-coarse}
 a_k(\cC)\le C_k\sqrt n\sigma(\cC)^{k-1}.
\end{align}
\end{enumerate}
\end{lemma}

\begin{proof}
The first assertion follows from the first term in the maximum in \eqref{eq:recursive-ak}.

We prove the second assertion by induction on $k$. Let $\iota:\cA\to\cC$ witness $\cA\preccurlyeq\cC$. Since the blocks of each partition are disjoint and $A\subseteq\iota(A)$,
$$
N(\cA)\le N(\cC).
$$
Thus \eqref{eq:a-restriction} holds for $k = 1$. Assume it holds for $k-1$. For every $A\in\cA$, the restriction of $\iota$ to $\cA\setminus\{A\}$ witnesses
$$
\cA\setminus\{A\}\preccurlyeq\cC\setminus\{\iota(A)\}.
$$
The induction hypothesis and $\abs A\le\abs{\iota(A)}$ give
$$
\sqrt{\abs A} a_{k-1}(\cA\setminus\{A\})\le\sqrt{\abs{\iota(A)}} a_{k-1}(\cC\setminus\{\iota(A)\}).
$$
Summing over $A\in\cA$ and using injectivity of $\iota$ bounds the weighted sum for $\cA$ by a sub-sum of the weighted sum for $\cC$. The first term in \eqref{eq:recursive-ak} is controlled by the induction hypothesis. Hence \eqref{eq:a-restriction} holds at level $k$.

We prove \eqref{eq:a-weighted} by induction on $k$, simultaneously for all partitions with at least $k$ blocks. For $k = 1$, $W_0(\cC) = \sqrt{N(\cC)}\ge1$, and therefore
$$
 a_1(\cC)\le4\sqrt2W_0(\cC).
$$
Assume the assertion at level $k-1$, and let $\#\cC\ge k$. Every partition $\cC\setminus\{B\}$ has at least $k-1$ blocks, so the induction hypothesis and \eqref{eq:W-deletion} give
$$
\sum_{B\in\cC}\sqrt{\abs B} a_{k-1}(\cC\setminus\{B\})\le C_{k-1}\sum_{B\in\cC}\sqrt{\abs B} W_{k-2}(\cC\setminus\{B\}) = C_{k-1}(k-1)W_{k-1}(\cC).
$$
The induction hypothesis and \eqref{eq:W-monotone} also give
$$
 a_{k-1}(\cC)\le C_{k-1}W_{k-2}(\cC)\le C_{k-1}(k-1)W_{k-1}(\cC).
$$
By Lemma~\ref{lem:weighted-identities}(1), $W_{k-1}(\cC)\ge1$. Hence \eqref{eq:recursive-ak} yields
$$
 a_k(\cC)\le\left(1+2\sqrt2(k-1)C_{k-1}\right)W_{k-1}(\cC).
$$
Starting with $C_1 = 4\sqrt2$ and defining $C_k = 1+2\sqrt2(k-1)C_{k-1}$ proves \eqref{eq:a-weighted}. The coarse estimate \eqref{eq:a-coarse} follows from \eqref{eq:W-coarse} with $j = k-1$.
\end{proof}

\section{Proof of the small-ball theorem} \label{sec:smallball}

We now prove the zero-centered estimate. After normalization, a point of a nonempty sublevel set and a vertex extremizer determine one fixed block and one fixed sign contraction. The weights in Lemma~\ref{lem:block-slope} are chosen from the residual recursive scales. Affine slicing bounds every selected-block fiber in terms of the same scalar contraction. If the contraction is nonconstant, inactive variables are deleted and the induction hypothesis is applied with the inherited partition. The residual scale appearing in the induction hypothesis then cancels with the same scale in the weighted slope lower bound. The exact recursion \eqref{eq:Phi-recursion} performs the final integration. A constant contraction is handled directly.

\begin{proposition}[Zero-centered profile] \label{prop:zero-centered}
Let $P:[0,1]^n\to\R$ be a nonzero multi-affine polynomial of exact degree $r\ge1$, assume that all $n$ variables are active, and let $\cB\in\Adm(P)$. Then
\begin{align} \label{eq:zero-centered}
\mu_n\{\abs P\le\rho\}\le\Phi_r\left( a_r(\cB)\frac{\rho}{\norm P_\infty}\right),\quad \rho\ge0.
\end{align}
\end{proposition}

\begin{proof}
We induct on $r$ and normalize $\norm P_\infty = 1$. The case $\rho = 0$ follows from Lemma~\ref{lem:zero-set}, and the case $r = 1$ is Proposition~\ref{prop:affine-base}.

Assume $r\ge2$ and the claim in every smaller positive degree. Since $P$ has exact degree $r$ and $\cB$ is admissible, $\#\cB\ge r$.

\emph{The nontrivial range.} Put $\lambda = a_r(\cB)$. If $\lambda\rho\ge1$, the right-hand side of \eqref{eq:zero-centered} is one. We may therefore assume
\begin{align} \label{eq:nontrivial}
\lambda\rho<1.
\end{align}
If the sublevel set is empty, there is nothing to prove.

\emph{Weighted selection of a fixed contraction.} Choose one point $x$ with $\abs{P(x)}\le\rho$. By Lemma~\ref{lem:vertex}, there is a vertex $z$ with $\abs{P(z)} = 1$, and therefore
\begin{align} \label{eq:oscillation-gap-smallball}
\abs{P(z)-P(x)}\ge1-\rho.
\end{align}
For $C\in\cB$, define the positive weight
$$
c_C = \sqrt{\abs C} a_{r-1}(\cB\setminus\{C\}),
$$
and put
\begin{align} \label{eq:S-recursive}
S_{r-1}(\cB) = \sum_{C\in\cB}\sqrt{\abs C} a_{r-1}(\cB\setminus\{C\}).
\end{align}
Lemma~\ref{lem:block-slope}, applied with these weights, gives a block $B\in\cB$ and $w\in\{-1,1\}^B$ such that, with $b = \abs B$ and $Q = Q_{B,w}$,
\begin{align} \label{eq:Q-lower-smallball}
\norm Q_\infty\ge\frac{\sqrt b a_{r-1}(\cB\setminus\{B\})}{S_{r-1}(\cB)}(1-\rho)>0.
\end{align}
The point $x$ is used only to select $B$ and $w$. From now on, $B,w,Q$ remain fixed on every complementary fiber.

\emph{Fiber integration.} Let $J = [n]\setminus B$. From \eqref{eq:block-representation},
$$
P(y,z) = R_B(y)+a_B(y)\cdot z,\quad y\in[0,1]^J,\quad z\in[0,1]^B,
$$
and $Q(y) = a_B(y)\cdot w$. Fubini and Lemma~\ref{lem:block-fiber} give
\begin{align} \label{eq:fiber-integral}
\mu_n\{\abs P\le\rho\}\le\int_{[0,1]^J}\min\left\{1,\frac{2\sqrt{2b}\rho}{\abs{Q(y)}}\right\}d\mu_J(y).
\end{align}
By Lemma~\ref{lem:layer-cake}, the right-hand side is at most
\begin{align} \label{eq:layer-Q}
\int_0^1\mu_J\left\{\abs Q\le\frac{2\sqrt{2b}\rho}{t}\right\}dt.
\end{align}
The layer-cake identity first produces a strict inequality inside the sublevel event. If $Q$ is nonconstant, Lemma~\ref{lem:zero-set} shows that the strict and non-strict events have the same $\mu_J$-measure. If $Q$ is constant, then $Q\ne0$ by \eqref{eq:Q-lower-smallball}, and the two events can differ for at most one value of $t$. Thus the replacement does not change the $dt$-integral in \eqref{eq:layer-Q}.

\emph{The constant-contraction case.} Suppose first that $Q$ is a nonzero constant. Then \eqref{eq:fiber-integral} is at most
$$
\Phi_1\left(\frac{2\sqrt{2b}\rho}{\norm Q_\infty}\right).
$$
Since $ a_{r-1}(\cB\setminus\{B\})\ge1$, \eqref{eq:Q-lower-smallball} gives
$$
\frac{2\sqrt{2b}\rho}{\norm Q_\infty}\le\frac{2\sqrt2S_{r-1}(\cB)\rho}{1-\rho}.
$$
By \eqref{eq:recursive-ak},
$$
\lambda-1\ge2\sqrt2S_{r-1}(\cB).
$$
Therefore
$$
\frac{2\sqrt{2b}\rho}{\norm Q_\infty}\le\frac{(\lambda-1)\rho}{1-\rho}\le\lambda\rho,
$$
where the last inequality is equivalent to \eqref{eq:nontrivial}. Since $\Phi_1\le\Phi_r$, the result follows in this case.

\emph{The nonconstant-contraction case.} Now suppose that $Q$ is nonconstant. Let $s = \deg Q\in\{1,\ldots,r-1\}$, let $A = A_{B,w}\subseteq J$ be its active-variable set, and let $\cC = \cB_{B,w}$ be the inherited admissible partition of $A$. Since $Q$ depends only on $y_A$ and $\mu_J = \mu_A\otimes\mu_{J\setminus A}$,
$$
\mu_J\{\abs Q\le u\} = \mu_A\{\abs Q\le u\},\quad u\ge0.
$$
Thus integrating the inactive coordinates contributes the factor one. We replace the ambient cube by $[0,1]^A$, the product measure by $\mu_A$, and the remaining blocks by $\cC$. The supremum norm of $Q$ is unchanged. Since $Q$ has exact degree $s$ and $\cC\in\Adm(Q)$, we have $\#\cC\ge s$, so the induction hypothesis applies to $Q$ on $[0,1]^A$. Lemma~\ref{lem:inherited} and Lemma~\ref{lem:recursive-scale} give
\begin{align} \label{eq:Q-scale-monotone}
 a_s(\cC)\le a_{r-1}(\cC)\le a_{r-1}(\cB\setminus\{B\}).
\end{align}
The induction hypothesis and $\Phi_s\le\Phi_{r-1}$ therefore imply, for $t\in(0,1)$,
$$
\mu_A\left\{\abs Q\le\frac{2\sqrt{2b}\rho}{t}\right\}\le\Phi_{r-1}\left( a_{r-1}(\cB\setminus\{B\})\frac{2\sqrt{2b}\rho}{t\norm Q_\infty}\right).
$$
Insert this estimate into \eqref{eq:layer-Q} and use \eqref{eq:Phi-recursion}:
\begin{align} \label{eq:profile-after-integration}
\mu_n\{\abs P\le\rho\}\le\Phi_r\left( a_{r-1}(\cB\setminus\{B\})\frac{2\sqrt{2b}\rho}{\norm Q_\infty}\right).
\end{align}
The weighted lower bound \eqref{eq:Q-lower-smallball} cancels the child scale and gives
$$
 a_{r-1}(\cB\setminus\{B\})\frac{2\sqrt{2b}\rho}{\norm Q_\infty}\le\frac{2\sqrt2S_{r-1}(\cB)\rho}{1-\rho}\le\frac{(\lambda-1)\rho}{1-\rho}\le\lambda\rho.
$$
This closes the induction.
\end{proof}

\begin{corollary}[All-center recursive estimate] \label{cor:recursive-allcenter}
Let $P:[0,1]^n\to\R$ be a nonconstant multi-affine polynomial of exact degree $r\ge1$, assume that all variables are active, and let $\cB\in\Adm(P)$. Then
\begin{align} \label{eq:recursive-allcenter}
\cQ_P(\rho)\le\Phi_r\left(2 a_r(\cB)\frac{\rho}{\osc(P)}\right),\quad \rho\ge0.
\end{align}
\end{corollary}

\begin{proof}
Fix $u\in\R$. Subtracting a constant preserves the active variables, the exact degree, and every admissible partition. Proposition~\ref{prop:zero-centered} gives
$$
\mu_n\{\abs{P-u}\le\rho\}\le\Phi_r\left( a_r(\cB)\frac{\rho}{\norm{P-u}_\infty}\right).
$$
If the range of $P$ is $[m,M]$, then
$$
\norm{P-u}_\infty = \max\{\abs{M-u},\abs{m-u}\}\ge\frac{M-m}{2} = \frac{\osc(P)}2.
$$
Hence
$$
\mu_n\{\abs{P-u}\le\rho\}\le\Phi_r\left(2 a_r(\cB)\frac{\rho}{\osc(P)}\right).
$$
Taking the supremum over $u$ proves \eqref{eq:recursive-allcenter}.
\end{proof}

\begin{proof}[Proof of Theorem~\ref{thm:main-smallball}]
Corollary~\ref{cor:recursive-allcenter} and Lemma~\ref{lem:recursive-scale}(3) give
$$
\cQ_P(\rho)\le\Phi_d\left(C_dw_d(\cB)\frac{\rho}{\osc(P)}\right),
$$
which is \eqref{eq:main-weighted}.
\end{proof}

\section{Proof of the block-product theorem} \label{sec:models}

The monomial $M_d(x) = x_1\cdots x_d$ is the exact prototype: Lemma~\ref{lem:Phi-product} gives its distribution function $\Phi_d$ and its density \eqref{eq:uniform-product-density}. The purpose of this section is to show that the same profile persists uniformly for products of high-dimensional centered block averages and to identify the resulting dimensional obstruction.

\subsection{Multiplicative tensorization}

\begin{lemma}[Multiplicative tensorization] \label{lem:tensorization}
Let $Y_1,\ldots,Y_r$ be independent real random variables. Suppose that for $d_j\ge1$, $M_j>0$, and $a_j\ge1$,
\begin{align} \label{eq:tensor-assumption}
\Prob\{\abs{Y_j}\le tM_j\}\le\Phi_{d_j}(a_jt), \quad t\ge0.
\end{align}
Put $D = d_1+\cdots+d_r$ and $A = a_1\cdots a_r$. Then
\begin{align} \label{eq:tensor-conclusion}
\Prob\left\{\left|\prod_{j = 1}^rY_j\right| \le t\prod_{j = 1}^rM_j\right\} \le\Phi_D(At), \quad t\ge0.
\end{align}
\end{lemma}

\begin{proof}
The assumption at $t = 0$ gives $\Prob\{Y_j = 0\} = 0$, so the logarithms below are well defined almost surely. Put
$$
V_j = -\log\frac{\abs{Y_j}}{M_j}.
$$
For every $u\in\R$,
$$
\Prob\{V_j\ge u\} = \Prob\{\abs{Y_j}\le e^{-u}M_j\}\le\Phi_{d_j}(a_je^{-u}).
$$
Let $G_{d_j}^{(j)}$ be independent gamma random variables with shape $d_j$ and rate one, and set $W_j = \log a_j+G_{d_j}^{(j)}$. By Lemma~\ref{lem:Phi-product},
$$
\Prob\{W_j\ge u\} = \Prob\{G_{d_j}^{(j)}\ge u-\log a_j\} = \Phi_{d_j}(a_je^{-u}),
$$
where both sides equal one when $u\le\log a_j$. Thus $\Prob\{V_j\ge u\}\le\Prob\{W_j\ge u\}$ for every $u$. Since $W_j$ has a continuous distribution,
$$
\Prob\{V_j>u\}\le\Prob\{V_j\ge u\}\le\Prob\{W_j\ge u\} = \Prob\{W_j>u\}.
$$
Hence $F_{V_j}(u)\ge F_{W_j}(u)$ for every $u$, where $F_Z(u) = \Prob\{Z\le u\}$, so $V_j$ is stochastically dominated by $W_j$.

For a distribution function $F$, write $F^{-1}(v) = \inf\{t\in\R:F(t)\ge v\}$ for $0<v<1$. For each $j$, let $U_j$ be an independent uniform random variable on $(0,1)$ and define $V_j' = F_{V_j}^{-1}(U_j)$ and $W_j' = F_{W_j}^{-1}(U_j)$. The stochastic domination implies $V_j'\le W_j'$ almost surely. Because the $U_j$ are independent, the pairs $(V_j',W_j')$ are independent across $j$, and the vectors $(V_1',\ldots,V_r')$ and $(W_1',\ldots,W_r')$ have the same laws as $(V_1,\ldots,V_r)$ and $(W_1,\ldots,W_r)$, respectively.

The sum of independent gamma variables with common rate one is gamma with the sum of the shapes. Hence
$$
\sum_{j = 1}^rW_j'\stackrel{\mathrm{law}}{ = }\log A+G_D,
$$
where $G_D$ has gamma distribution with shape $D = d_1+\cdots+d_r$ and rate one. For $t>0$,
$$
\left\{\left|\prod_{j = 1}^rY_j\right|\le t\prod_{j = 1}^rM_j\right\} = \left\{\sum_{j = 1}^rV_j\ge\log(1/t)\right\}.
$$
The coupling therefore gives
$$
\Prob\left\{\sum_{j = 1}^rV_j\ge\log(1/t)\right\}\le\Prob\{\log A+G_D\ge\log(1/t)\} = \Phi_D(At).
$$
The case $t = 0$ follows from $Y_j\ne0$ almost surely. This proves \eqref{eq:tensor-conclusion}.
\end{proof}

\subsection{Centered block averages}

For $m\ge1$, define
\begin{align} \label{eq:A-m}
A_m(x) = \frac2m\sum_{i = 1}^m x_i-1, \quad x\in[0,1]^m.
\end{align}
Its coefficient vector has Euclidean norm $2/\sqrt m$. If $X^{(m)}\sim\mu_m$, Proposition~\ref{prop:affine-slicing} gives
\begin{align} \label{eq:A-m-smallball}
\Prob\{\abs{A_m(X^{(m)})}\le t\}\le\Phi_1(\sqrt{2m}t),\quad t\ge0.
\end{align}

We also use the following elementary one-dimensional fact.

\begin{lemma}[Central lower bound] \label{lem:central-density}
There are universal constants $a,b>0$ such that every even log-concave probability density $f$ on $\R$ with variance $1/3$ satisfies
\begin{align} \label{eq:central-density}
f(x)\ge b, \quad \abs x\le a.
\end{align}
\end{lemma}

\begin{proof}[Proof of Lemma~\ref{lem:central-density}]
Let $Z$ be a random variable with density $f$. We prove that the universal choices $a = 1/64$ and $b = 1/6$ are admissible. Replacing $f$ at the endpoints of its support by its upper semicontinuous log-concave representative does not change the probability measure. Since $f$ is an integrable nonzero log-concave density with positive finite variance, its upper semicontinuous representative satisfies $0<f(0)<\infty$. Since $f$ is even and log-concave, it is nonincreasing on $[0,\infty)$. Put $M = f(0)$, and use the extended-real convention $\log0 = -\infty$.

Chebyshev's inequality and $\E Z^2 = 1/3$ give
$$
\Prob\{\abs Z>1\}\le\frac13,\quad \Prob\{\abs Z\le1\}\ge\frac23.
$$
Since $f(x)\le M$ for every $x$,
$$
\frac23\le\int_{-1}^1f(x)dx\le2M,
$$
and hence
\begin{align} \label{eq:M-lower-app}
M\ge\frac13.
\end{align}

Define
$$
r = \sup\{x\ge0:f(x)>M/e\}.
$$
The set is nonempty because $f(0) = M$, and $r<\infty$ because $f$ is integrable. Since $f(x)>M/e$ for $0\le x<r$,
$$
\frac12 = \int_0^\infty f(x)dx\ge\int_0^r f(x)dx\ge\frac{Mr}{e}.
$$
Thus $r\le e/(2M)<3/(2M)$.

Fix $\varepsilon>0$. By the definition of $r$, $f(r+\varepsilon)\le M/e$. Concavity of $\log f$ on the support implies that for $x\ge r+\varepsilon$,
$$
f(x)\le M\exp\left(-\frac{x}{r+\varepsilon}\right).
$$
If $f$ vanishes before $x$, the same estimate is automatic. Using $f\le M$ on $[0,r+\varepsilon]$ and the change of variables $u = x/(r+\varepsilon)$ on the tail,
$$
\begin{aligned}
\frac13 = \E Z^2& = 2\int_0^\infty x^2f(x)dx \\
&\le2M\int_0^{r+\varepsilon}x^2dx+2M\int_{r+\varepsilon}^\infty x^2e^{-x/(r+\varepsilon)}dx \\
& = 2M(r+\varepsilon)^3\left(\frac13+\int_1^\infty u^2e^{-u}du\right)<5M(r+\varepsilon)^3.
\end{aligned}
$$
Letting $\varepsilon\downarrow0$ and using $r<3/(2M)$ gives
$$
\frac13<5Mr^3<\frac{135}{8M^2}.
$$
Hence $M^2<405/8<64$, so $M<8$. Set $a = 1/64$. Suppose, toward a contradiction, that $f(a)<M/2$. Since $\log f$ is concave and $f(0) = M$, the slopes of its secant lines from the origin are nonincreasing. Therefore, for every integer $k\ge1$,
$$
\log f(ka)-\log M\le k(\log f(a)-\log M)<-k\log2,
$$
and hence $f(ka)\le M2^{-k}$. By monotonicity on $[0,\infty)$,
$$
\frac12 = \int_0^\infty f(x)dx\le Ma+\sum_{k = 1}^\infty\int_{ka}^{(k+1)a}f(x)dx\le Ma+aM\sum_{k = 1}^\infty2^{-k} = 2Ma<\frac14,
$$
a contradiction. Thus $f(a)\ge M/2$. Evenness and monotonicity now imply, for $\abs x\le a$,
$$
f(x)\ge f(a)\ge\frac M2\ge\frac16.
$$
This proves \eqref{eq:central-density} with $a = 1/64$ and $b = 1/6$.
\end{proof}

\subsection{Proof of the block-product theorem}

\begin{proof}[Proof of Theorem~\ref{thm:block-products}]
Let $X\sim\mu_{N_{\mathbf m}}$. For each $k$, the factor
$$
A_{m_k}(x_{B_k}) = \frac2{m_k}\sum_{i\in B_k}x_i-1
$$
takes values in $[-1,1]$ and attains both endpoints. Hence $\norm{P_{\mathbf m,d}}_\infty = 1$. By choosing the signs of the factors so that their product is $1$ or $-1$, we also obtain $\osc(P_{\mathbf m,d}) = 2$. Expanding the product shows that the coefficient of every monomial obtained by choosing one variable from each block is nonzero, so the polynomial has exact degree $d$. Every variable occurs in such a monomial and is therefore active. Since each monomial uses at most one variable from each block, the natural partition $\{B_1,\ldots,B_d\}$ is admissible.

For the upper bound, let $Y_k = A_{m_k}(X_{B_k})$ and $M_k = m_k^{-1/2}$. The random vectors $X_{B_1},\ldots,X_{B_d}$ are independent, so the variables $Y_1,\ldots,Y_d$ are independent. Equation \eqref{eq:A-m-smallball} gives
$$
\Prob\{\abs{Y_k}\le tM_k\}\le\Phi_1(\sqrt2t),\quad t\ge0.
$$
Applying Lemma~\ref{lem:tensorization} with $d_k = 1$ and $a_k = \sqrt2$ yields
$$
\Prob\left\{\left|\prod_{k = 1}^dY_k\right|\le\frac{s}{M_{\mathbf m}}\right\}\le\Phi_d(2^{d/2}s),\quad s\ge0.
$$
This proves the asserted upper estimate.

For the lower bound, define
$$
Z_{m_k,k} = \sqrt{m_k}Y_k = \frac1{\sqrt{m_k}}\sum_{i\in B_k}(2X_i-1).
$$
Each summand $2X_i-1$ is uniform on $[-1,1]$, has mean zero, and has variance $1/3$. Therefore $Z_{m_k,k}$ is centered and has variance
$$
\frac1{m_k}\sum_{i\in B_k}\frac13 = \frac13.
$$
Its density is even. It is also log-concave: for $m_k = 1$ this is the uniform density on $[-1,1]$, and for $m_k\ge2$ it follows from preservation of log-concavity under convolution and dilation \cite{Prekopa73}. Lemma~\ref{lem:central-density} therefore gives universal constants $a,b>0$ such that every $Z_{m_k,k}$ has density at least $b$ on $[-a,a]$.

The variables $Z_{m_1,1},\ldots,Z_{m_d,d}$ are independent. Hence, for $0\le s\le a^d$,
$$
\Prob\left\{\left|\prod_{k = 1}^dZ_{m_k,k}\right|\le s\right\}\ge b^d\vol_d\left\{z\in[-a,a]^d:\left|\prod_{k = 1}^dz_k\right|\le s\right\}.
$$
The set in braces is invariant under changes of signs. Its intersection with each orthant has the same volume, and the scaling $z_k = au_k$ maps its positive-orthant part onto
$$
\left\{u\in[0,1]^d:\prod_{k = 1}^du_k\le s/a^d\right\}.
$$
By Lemma~\ref{lem:Phi-product}, this set has volume $\Phi_d(s/a^d)$. Consequently,
$$
\vol_d\left\{z\in[-a,a]^d:\left|\prod_{k = 1}^dz_k\right|\le s\right\} = (2a)^d\Phi_d(s/a^d).
$$
Since
$$
P_{\mathbf m,d}(X) = M_{\mathbf m}^{-1}\prod_{k = 1}^dZ_{m_k,k},
$$
we obtain
$$
\mu_{N_{\mathbf m}}\left\{\abs{P_{\mathbf m,d}}\le\frac{s}{M_{\mathbf m}}\right\}\ge(2ab)^d\Phi_d(s/a^d),\quad 0\le s\le a^d.
$$
By the ratio limit \eqref{eq:Phi-ratio}, after decreasing the upper range of $s$ there is $c_d>0$ such that
$$
\Phi_d(c_ds)\le(2ab)^d\Phi_d(s/a^d),\quad 0\le s\le s_d.
$$
This proves the lower estimate in \eqref{eq:block-products-two-sided}.

We next prove the optimality of the full block-size vector. Let $\Gamma:\N^d\to(0,\infty)$, and suppose that for some $K_d>0$ every exact degree-$d$ multi-affine polynomial with all variables active and with a prescribed admissible $d$-block partition of sizes $m_1,\ldots,m_d$ satisfies
\begin{align} \label{eq:hypothetical-vector}
\cQ_P(\rho)\le\Phi_d\left(K_d\Gamma(\mathbf m)\frac{\rho}{\osc(P)}\right),\quad \rho\ge0.
\end{align}
Apply this estimate to $P_{\mathbf m,d}$ at radius $\rho = s/M_{\mathbf m}$. The lower estimate in \eqref{eq:block-products-two-sided} and \eqref{eq:block-product-osc} give
$$
\Phi_d(c_ds)\le\Phi_d\left(\frac{K_d\Gamma(\mathbf m)}{2M_{\mathbf m}}s\right),\quad 0<s\le s_d.
$$
Divide by $\Phi_d(s)$ and let $s\downarrow0$. Equation \eqref{eq:Phi-ratio} gives
$$
c_d\le\frac{K_d\Gamma(\mathbf m)}{2M_{\mathbf m}}.
$$
Thus $\Gamma(\mathbf m)\ge2c_dK_d^{-1}M_{\mathbf m}$, proving the stated optimality among uniform estimates with profile $\Phi_d$.

It remains to prove the dimensional obstruction. Take equal blocks $m_1 = \cdots = m_d = m$, write $P_{m,d} = P_{(m,\ldots,m),d}$, and note that $n = dm$. Fix $m$ and put $\rho = m^{-d/2}s$. The centered probability is bounded by the concentration function:
$$
\mu_{dm}\{\abs{P_{m,d}}\le\rho\}\le\cQ_{P_{m,d}}(\rho).
$$
Apply the hypothetical estimate \eqref{eq:hypothetical-alpha}. Since $\osc(P_{m,d}) = 2$, the lower estimate gives, for all sufficiently small $s>0$,
$$
\Phi_d(c_ds)\le\Phi_d\left(\frac{K_{d,q}}2(dm)^\alpha m^{-d/2}s\right).
$$
Divide by $\Phi_d(s)$ and let $s\downarrow0$. The ratio limit \eqref{eq:Phi-ratio} gives
$$
c_d\le\frac{K_{d,q}}2d^\alpha m^{\alpha-d/2}.
$$
This inequality holds for every positive integer $m$. Letting $m\to\infty$ is impossible when $\alpha<d/2$, and therefore $\alpha\ge d/2$.
\end{proof}

\section{Proof of the density theorem} \label{sec:density}

\subsection{Absolute continuity and negative moments}

\begin{proposition}[Absolute continuity] \label{prop:absolute-continuity}
The image of $\mu_n$ under every nonconstant multi-affine polynomial is absolutely continuous with respect to Lebesgue measure on $\R$.
\end{proposition}

\begin{proof}
Choose an active coordinate and relabel it as $x_n$. Write
$$
P(x',x_n) = R(x')+x_nQ(x'),
$$
where $Q$ is a nonzero multi-affine polynomial. By Lemma~\ref{lem:zero-set}, $Q(x')\ne0$ for almost every $x'$. Conditional on such an $x'$, the image of the uniform $x_n$ is uniform on the interval with endpoints $R(x')$ and $R(x')+Q(x')$. Its density is $1/\abs{Q(x')}$ on that interval and zero elsewhere, so it assigns zero mass to every Lebesgue-null set. Fubini proves the claim; the exceptional set $\{Q = 0\}$ is null.
\end{proof}

\begin{lemma} \label{lem:negative-moment-Q}
Let $A$ be a finite coordinate set, let $Q:[0,1]^A\to\R$ be a nonzero multi-affine polynomial of exact degree $r\ge1$ with all variables active, and let $\cC\in\Adm(Q)$. For $0<\eta<1$,
\begin{align} \label{eq:negative-moment-Q}
\int_{[0,1]^A}\abs{Q(y)}^{-\eta}d\mu_A(y)\le\left(\frac{ a_r(\cC)}{\norm Q_\infty}\right)^\eta(1-\eta)^{-r}.
\end{align}
\end{lemma}

\begin{proof}
Lemma~\ref{lem:zero-set} gives $Q(y)\ne0$ for $\mu_A$-almost every $y$. Define the positive random variable
$$
Z = a_r(\cC)\frac{\abs{Q(Y)}}{\norm Q_\infty},\quad Y\sim\mu_A.
$$
For $t\ge0$, Proposition~\ref{prop:zero-centered} gives
$$
\Prob\{Z\le t\} = \mu_A\left\{\abs Q\le\frac{t\norm Q_\infty}{ a_r(\cC)}\right\}\le\Phi_r(t).
$$
Lemma~\ref{lem:inverse-moments} therefore yields $\E Z^{-\eta}\le(1-\eta)^{-r}$. Multiplying by $( a_r(\cC)/\norm Q_\infty)^\eta$ proves \eqref{eq:negative-moment-Q}.
\end{proof}

\subsection{The quantitative estimate}

\begin{proposition}[Recursive density estimate] \label{prop:recursive-density}
Let $P:[0,1]^n\to\R$ be a nonconstant multi-affine polynomial of exact degree $r\ge1$, assume that all variables are active, and let $\cB\in\Adm(P)$. If $X\sim\mu_n$ and $f_P$ is the density of $P(X)$, then, for $1<p<\infty$,
\begin{align} \label{eq:recursive-density}
\norm{f_P}_{L^p(\R)}\le p^{r-1}\left(\frac{ a_r(\cB)}{\osc(P)}\right)^{1-1/p}.
\end{align}
\end{proposition}

\begin{proof}
Put $\eta = 1-1/p$. We first treat the affine case. If $r = 1$, Proposition~\ref{prop:affine-base} gives
$$
\norm{f_P}_{L^p(\R)}\le\left(\frac{\sqrt{2n}}{\osc(P)}\right)^\eta.
$$
Since $ a_1(\cB)\ge4\sqrt{2n}$, this is stronger than \eqref{eq:recursive-density}.

Assume $r\ge2$. Since $P$ has exact degree $r$ and $\cB$ is admissible, $\#\cB\ge r$. Define $S_{r-1}(\cB)$ by \eqref{eq:S-recursive}. Choose points $x,z\in[0,1]^n$ at which $P$ attains its minimum and maximum. Apply Lemma~\ref{lem:block-slope} with the weights
$$
c_C = \sqrt{\abs C} a_{r-1}(\cB\setminus\{C\}),\quad C\in\cB.
$$
There are a block $B\in\cB$ and a sign vector $w\in\{-1,1\}^B$ such that, with $b = \abs B$ and $Q = Q_{B,w}$,
\begin{align} \label{eq:Q-lower-density}
\norm Q_\infty\ge\frac{\sqrt b a_{r-1}(\cB\setminus\{B\})}{S_{r-1}(\cB)}\osc(P)>0.
\end{align}
The block $B$, the sign vector $w$, and the contraction $Q$ remain fixed below.

Let $J = [n]\setminus B$. In the notation of \eqref{eq:block-representation},
$$
P(y,z) = R_B(y)+a_B(y)\cdot z,\quad Q(y) = a_B(y)\cdot w,\quad y\in[0,1]^J,\quad z\in[0,1]^B.
$$
Because $Q$ is a nonzero multi-affine polynomial, Lemma~\ref{lem:zero-set} gives $Q(y)\ne0$ for $\mu_J$-almost every $y$. The implication $a_B(y) = 0\Rightarrow Q(y) = 0$ shows that the affine fiber $z\mapsto P(y,z)$ is nonconstant for almost every $y$.

Set
$$
I_Q = \int_{[0,1]^J}\abs{Q(y)}^{-\eta}d\mu_J(y).
$$
If $Q$ is a nonzero constant, then $I_Q = \abs Q^{-\eta}<\infty$. Suppose that $Q$ is nonconstant. Let $A = A_{B,w}\subseteq J$ be its active-variable set, let $s = \deg Q\le r-1$, and let $\cC = \cB_{B,w}$ be the inherited admissible partition of $A$. Since $Q$ depends only on $y_A$ and $\mu_J = \mu_A\otimes\mu_{J\setminus A}$,
$$
I_Q = \int_{[0,1]^A}\abs{Q(y_A)}^{-\eta}d\mu_A(y_A).
$$
Lemma~\ref{lem:negative-moment-Q} gives
\begin{align} \label{eq:density-negative-moment}
I_Q\le\left(\frac{ a_s(\cC)}{\norm Q_\infty}\right)^\eta(1-\eta)^{-s}<\infty.
\end{align}
Thus $I_Q<\infty$ in both cases.

Apply Lemma~\ref{lem:disintegration} with $\mathsf Y = [0,1]^J$, $\nu = \mu_J$, $R = R_B$, and $a = a_B$, and let $h_y$ be the resulting conditional density. By \eqref{eq:block-fiber-Lp},
$$
\int_{[0,1]^J}\norm{h_y}_{L^p(\R)}d\mu_J(y)\le(\sqrt{2b})^\eta I_Q<\infty.
$$
The finiteness condition in Lemma~\ref{lem:disintegration} is therefore satisfied, and Minkowski's integral inequality gives
\begin{align} \label{eq:density-Minkowski}
\norm{f_P}_{L^p(\R)}\le\int_{[0,1]^J}\norm{h_y}_{L^p(\R)}d\mu_J(y)\le(\sqrt{2b})^\eta I_Q.
\end{align}

Suppose first that $Q$ is a nonzero constant. By \eqref{eq:Q-lower-density} and $ a_{r-1}(\cB\setminus\{B\})\ge1$,
$$
\frac{\sqrt{2b}}{\norm Q_\infty}\le\frac{\sqrt2S_{r-1}(\cB)}{\osc(P)}.
$$
Equation \eqref{eq:recursive-ak} gives $ a_r(\cB)\ge1+2\sqrt2S_{r-1}(\cB)\ge\sqrt2S_{r-1}(\cB)$. Hence \eqref{eq:density-Minkowski} proves \eqref{eq:recursive-density} in the constant case.

Now assume that $Q$ is nonconstant. By Lemma~\ref{lem:recursive-scale} and $s\le r-1$,
$$
 a_s(\cC)\le a_{r-1}(\cC)\le a_{r-1}(\cB\setminus\{B\}).
$$
Moreover, $1-\eta = 1/p$, so $(1-\eta)^{-s} = p^s\le p^{r-1}$. Combining \eqref{eq:density-Minkowski}, \eqref{eq:density-negative-moment}, and \eqref{eq:Q-lower-density}, the child scale cancels and gives
$$
\norm{f_P}_{L^p(\R)}\le p^{r-1}\left(\frac{\sqrt2S_{r-1}(\cB)}{\osc(P)}\right)^\eta.
$$
Since $ a_r(\cB)\ge1+2\sqrt2S_{r-1}(\cB)\ge\sqrt2S_{r-1}(\cB)$, this proves \eqref{eq:recursive-density}.
\end{proof}

\begin{proof}[Proof of Theorem~\ref{thm:main-density}]
Absolute continuity is Proposition~\ref{prop:absolute-continuity}. The estimate \eqref{eq:main-density} follows from Proposition~\ref{prop:recursive-density} and Lemma~\ref{lem:recursive-scale}(3).
\end{proof}

\subsection{Optimal \texorpdfstring{$p$}{p}-growth and the rearrangement consequence}

For $M_d(x) = x_1\cdots x_d$, Lemma~\ref{lem:Phi-product} gives \eqref{eq:monomial-density-intro}. Therefore, for $p>1$,
\begin{align} \label{eq:monomial-Lp}
\norm{f_{M_d}}_{L^p(\R)}^p = \frac{\Gamma(p(d-1)+1)}{((d-1)!)^p}.
\end{align}
Indeed, the change of variables $u = \log(1/t)$ converts the $p$th power of the density into the gamma integral. Stirling's formula gives
\begin{align} \label{eq:monomial-Lp-asymp}
\norm{f_{M_d}}_{L^p(\R)}\asymp_dp^{d-1}, \quad p\ge2.
\end{align}
For $d\ge2$, the density is not bounded.

\begin{proof}[Proof of Proposition~\ref{prop:block-density}]
The natural admissible partition of $P_{\mathbf m,d}$ has exactly $d$ blocks. Hence
$$
 w_d(\{B_1,\ldots,B_d\}) = dM_{\mathbf m}.
$$
The upper estimate in \eqref{eq:block-density} follows from Theorem~\ref{thm:main-density} and $\osc(P_{\mathbf m,d}) = 2$.

For the lower estimate, let $p\ge2$ and put $\eta = 1-1/p$. By H\"older's inequality, for every $\rho>0$,
\begin{align} \label{eq:Holder-density-lower}
\Prob\{\abs{P_{\mathbf m,d}(X)}\le\rho\} = \int_{-\rho}^{\rho}f_{\mathbf m,d}(t)dt\le(2\rho)^\eta\norm{f_{\mathbf m,d}}_{L^p(\R)}.
\end{align}
By \eqref{eq:Phi-asymp-intro}, after decreasing $s_d$ there is $b_d>0$ such that
$$
\Phi_d(c_ds)\ge b_ds\log^{d-1}(1/s),\quad 0<s\le s_d,
$$
with the convention $\log^0(1/s) = 1$ when $d = 1$. Choose $A_d>0$ so large that $e^{-2A_d}\le s_d$, and set $s = e^{-A_dp}$. Then $s\le s_d$, and Theorem~\ref{thm:block-products} with $\rho = s/M_{\mathbf m}$ gives
$$
\Prob\left\{\abs{P_{\mathbf m,d}(X)}\le\frac{s}{M_{\mathbf m}}\right\}\ge b_ds\log^{d-1}(1/s).
$$
Using this in \eqref{eq:Holder-density-lower} yields
$$
\norm{f_{\mathbf m,d}}_{L^p(\R)}\ge c_dM_{\mathbf m}^\eta s^{1/p}\log^{d-1}(1/s) = c_de^{-A_d}A_d^{d-1}p^{d-1}M_{\mathbf m}^\eta.
$$
Absorbing the fixed degree-dependent factor proves the lower estimate.
\end{proof}

\begin{proof}[Proof of Corollary~\ref{cor:rearrangement}]
A change of variables gives
$$
\norm g_{L^p(\R)} = K^{1/p-1}\norm{f_P}_{L^p(\R)}\le C_dp^{d-1},\quad p>1.
$$
For a nonnegative function, Chebyshev's inequality gives $g^{*}(s)\le s^{-1/p}\norm g_{L^p(\R)}$. If $0<s<1$, choose $p = \log(e/s)>1$. Then $s^{-1/p}\le e$, which proves \eqref{eq:main-rearrangement}. For $s = 1$, use $g^{*}(1)\le\norm g_{L^1(\R)} = 1$. The monomial density \eqref{eq:monomial-density-intro} shows that the logarithmic power cannot be reduced.
\end{proof}

\section{Proof of the Remez application} \label{sec:remez}

\subsection{The quotient-form inequality}

\begin{proof}[Proof of Application~\ref{thm:main-remez}]
Let
$$
\alpha = \essinf_\Omega P,\quad \beta = \esssup_\Omega P.
$$
For every $c\in\R$,
$$
\esssup_\Omega\abs{P-c}\ge\max\{\beta-c,c-\alpha\}\ge\frac{\beta-\alpha}{2}.
$$
For the midpoint $c = (\alpha+\beta)/2$, the inequalities $\alpha\le P\le\beta$ hold almost everywhere on $\Omega$, so $\abs{P-c}\le(\beta-\alpha)/2$ almost everywhere. Thus equality is attained and
\begin{align} \label{eq:rOmega-midpoint}
r_\Omega(P) = \frac12\left(\esssup_\Omega P-\essinf_\Omega P\right).
\end{align}
Since $\theta>0$, Lemma~\ref{lem:zero-set} implies $r_\Omega(P)>0$: otherwise $P = c$ almost everywhere on a set of positive measure. Up to a null set,
$$
\Omega\subseteq\{\abs{P-c}\le r_\Omega(P)\}.
$$
Theorem~\ref{thm:main-smallball} therefore gives
$$
\theta\le\Phi_d\left(C_d w_d(\cB)\frac{r_\Omega(P)}{\osc(P)}\right).
$$
Inverting $\Phi_d$ proves \eqref{eq:main-remez}. Lemma~\ref{lem:inverse-Phi} gives \eqref{eq:main-remez-log}. Equation \eqref{eq:w-coarse-intro} gives the $\sqrt n\sigma(\cB)^{d-1}$ form, and \eqref{eq:sigma-cs} gives the fixed-block form. If $\cB = \{B_1,\ldots,B_d\}$, the identity $ w_d(\cB) = d\prod_{j = 1}^d\sqrt{\abs{B_j}}$ gives the estimate for exactly $d$ blocks.
\end{proof}

The quotient formulation is invariant under adding constants to $P$, which is the natural symmetry of an all-center concentration inequality.

\begin{proof}[Proof of Corollary~\ref{cor:remez-sup}]
Put $M_\Omega = \norm P_{L^\infty(\Omega,\mu_n)}$. Choose $x_0\in\Omega$ outside the null set on which the essential bound $\abs{P(x)}\le M_\Omega$ may fail. For every $y\in[0,1]^n$,
$$
\abs{P(y)}\le\abs{P(y)-P(x_0)}+\abs{P(x_0)}\le\osc(P)+M_\Omega.
$$
Also, taking $c = 0$ in \eqref{eq:rOmega} gives $r_\Omega(P)\le M_\Omega$. Let
$$
F = C_d w_d(\cB)\frac{\log^{d-1}(e/\theta)}{\theta}.
$$
Application~\ref{thm:main-remez} gives $\osc(P)\le FM_\Omega$. Since $ w_d(\cB)\ge1$ by Lemma~\ref{lem:weighted-identities}(1) and $\log(e/\theta)/\theta\ge1$, we may enlarge $C_d$ so that $F\ge1$. Hence
$$
\norm P_{L^\infty([0,1]^n)}\le(F+1)M_\Omega\le2FM_\Omega,
$$
which is \eqref{eq:main-remez-sup} after changing the degree-dependent constant.
\end{proof}

\subsection{The lower model}

\begin{proof}[Proof of Proposition~\ref{prop:remez-model}]
Equation \eqref{eq:theta-two-sided} follows from Theorem~\ref{thm:block-products}, after renaming the degree-dependent constants. Choosing the center zero in \eqref{eq:rOmega} gives
$$
r_{\Omega_{\mathbf m,s}}(P_{\mathbf m,d})\le\frac{s}{M_{\mathbf m}},
$$
while $\osc(P_{\mathbf m,d}) = 2$. The lower inequality in \eqref{eq:theta-two-sided} implies $\Phi_d^{-1}(\theta_{\mathbf m,s})\ge a_ds$, and hence
$$
\frac{\osc(P_{\mathbf m,d})}{r_{\Omega_{\mathbf m,s}}(P_{\mathbf m,d})}\ge\frac{2M_{\mathbf m}}s\ge2a_d\frac{M_{\mathbf m}}{\Phi_d^{-1}(\theta_{\mathbf m,s})}.
$$
Thus \eqref{eq:remez-model} holds with $\kappa_d = 2a_d$.
\end{proof}

\section*{Acknowledgments}

The authors are grateful to Sasha Volberg for reading earlier drafts and for valuable remarks and comments.

AI-assisted tools, including ChatGPT, were used for literature searches, consistency and editorial checks, and in developing a lower-bound example. The authors are responsible for all mathematical arguments, references, and the final text.

\end{document}